\documentclass[english]{amsart}
\usepackage{babel}
\usepackage{dsfont}
\usepackage{amsmath,amsthm,amssymb,babel,amsfonts,graphics}
\usepackage{caption}
\usepackage{subcaption}
\usepackage{tikz}

\usepackage[pdftex,colorlinks]{hyperref}

\newcommand{\kom}[1]{}
\renewcommand{\kom}[1]{{\bf [#1]}}

\def\vint_#1{\mathchoice%
          {\mathop{\kern 0.2em\vrule width 0.6em height 0.69678ex depth -0.58065ex
                  \kern -0.8em \intop}\nolimits_{\kern -0.4em#1}}%
          {\mathop{\kern 0.1em\vrule width 0.5em height 0.69678ex depth -0.60387ex
                  \kern -0.6em \intop}\nolimits_{#1}}%
          {\mathop{\kern 0.1em\vrule width 0.5em height 0.69678ex
              depth -0.60387ex
                  \kern -0.6em \intop}\nolimits_{#1}}%
          {\mathop{\kern 0.1em\vrule width 0.5em height 0.69678ex depth -0.60387ex
                  \kern -0.6em \intop}\nolimits_{#1}}}
                  
                  \newcommand{\aveint}[2]{\mathchoice%
          {\mathop{\kern 0.2em\vrule width 0.6em height 0.69678ex depth -0.58065ex
                  \kern -0.8em \intop}\nolimits_{\kern -0.45em#1}^{#2}}%
          {\mathop{\kern 0.1em\vrule width 0.5em height 0.69678ex depth -0.60387ex
                  \kern -0.6em \intop}\nolimits_{#1}^{#2}}%
          {\mathop{\kern 0.1em\vrule width 0.5em height 0.69678ex depth -0.60387ex
                  \kern -0.6em \intop}\nolimits_{#1}^{#2}}%
          {\mathop{\kern 0.1em\vrule width 0.5em height 0.69678ex depth -0.60387ex
                  \kern -0.6em \intop}\nolimits_{#1}^{#2}}}

\usepackage{color}
\usepackage{epsfig}

\newcommand{\R}{\mathbb{R}}

\DeclareMathOperator{\Tr}{Tr}
\def\1{\raisebox{2pt}{\rm{$\chi$}}}

\newcommand{\dist}{\operatorname{dist}}

\newcommand{\kint}{\vint}

\newcommand{\eps}{\varepsilon}

\theoremstyle{plain}
\newtheorem{definition}{Definition}[section]
\newtheorem{proposition}[definition]{Proposition}
\newtheorem{theorem}[definition]{Theorem}
\newtheorem*{theorem*}{Main Theorem}
\newtheorem{corollary}[definition]{Corollary}
\newtheorem{lemma}[definition]{Lemma}

\newtheorem{remark}[definition]{Remark}

\theoremstyle{definition}

\theoremstyle{remark}

\numberwithin{equation}{section}

\begin{document}

\title[]{Tug-of-war games related to oblique derivative boundary value problems with the normalized $p$-Laplacian}

\author{Jeongmin Han}
\address{Department of Mathematics and Statistics, University of Jyv\"{a}skyl\"{a}, 
P.O. Box 35, FI-40014 Jyv\"{a}skyl\"{a}, Finland}
\email{jeongmin.han@pusan.ac.kr}



\begin{abstract}
In this paper, we are concerned with game-theoretic interpretations to the following oblique derivative boundary value problem
\begin{align*}
\left\{ \begin{array}{ll}
\Delta_{p}^{N}u=0 & \textrm{in $ \Omega$,}\\
\langle \beta , Du \rangle  + \gamma u = \gamma G & \textrm{on $ \partial \Omega$,}\\
\end{array} \right.
\end{align*}
where $\Delta_{p}^{N}$ is the normalized $p$-Laplacian.
This problem can be regarded as a generalized version of the Robin boundary value problem for the Laplace equations.
We construct several types of stochastic games associated with this problem by using `shrinking tug-of-war'.
For the value functions of such games, we investigate the properties such as existence, uniqueness, regularity and convergence.
\end{abstract}

\maketitle
\tableofcontents

\section{Introduction}

In this paper, we study the existence, uniqueness, local and boundary regularity of tug-of-war type games related to oblique boundary value problems involving the normalized $p$-Laplacian 
$$ \Delta_{p}^{N}u = \Delta u + (p-2)\Delta_{\infty}^{N}u= \Delta u + (p-2)
\frac{ \langle D^{2}uDu, Du \rangle}{|Du|^{2}} $$
for $2<p<\infty$. The value function of the game we  first study satisfies the following dynamic programming principle (DPP for short) in $\Omega\subset \R^n$ satisfying the $C^{1,1}$-boundary regularity condition,
\begin{align} \label{dppobp} \begin{split}
 u_{\eps}(x)&= (1- (1-\alpha)\gamma  s_{\eps}(x)) \bigg\{\frac{\alpha}{2}\bigg( \sup_{B_{\eps d'_{x}}(x)} u_{\eps}+  
\inf_{B_{\eps d'_{x}}(x)} u_{\eps} \bigg)\\
&\quad \quad + (1-\alpha)  \kint_{B_{\eps }(x)\cap \Omega}u_{\eps}(y) dy \bigg\}  + (1-\alpha)\gamma  s_{\eps}(x)G(x)
\end{split}
\end{align}
with $\gamma>0$ and $G\in C^{1}(\Gamma_{\eps})$ where 
$\Gamma_{\eps}$ is the inner $\eps$-boundary strip of $\Omega$ (see \eqref{estrip} for the definition).
For $x\in \Omega$ with
$s_{\eps}(x)=0 $,
the DPP reads
\[
 u_{\eps}(x)=\frac{\alpha}{2}\bigg( \sup_{B_{\frac{\eps}{2}}(x)} u_{\eps}+  
\inf_{B_{\frac{\eps}{2}}(x)} u_{\eps} \bigg) + (1-\alpha)  \kint_{B_{\eps }(x)}u_{\eps}(y) dy,
\]
which is related to the game tug-of-war with noise. 
We postpone a more detailed presentation of the various terms of the DPP \eqref{dppobp} to Section \ref{subsec:mvc}.

The games have a connection to the following boundary value problem when $\eps\rightarrow 0$,
\begin{align} \label{defcapf}
\left\{ \begin{array}{ll}
\Delta_{p}^{N}u=0 & \textrm{in $ \Omega$,}\\
\langle \mathbf{n}  , Du  \rangle + \gamma u = \gamma G & \textrm{on $ \partial \Omega$,}\\
\end{array} \right.
\end{align}
where 
$\mathbf{n}$ is  the outward normal vector on $\partial \Omega$.
 Theorems \ref{convrob} implies that the value function $u_\eps$ of DPP \eqref{dppobp} converges uniformly to a solution of this PDE when $p\geq 2$.
We give a heuristic game interpretation for the boundary condition of \eqref{defcapf}.
Our game is basically played only within $\Omega$, and hence we do not consider such a situation that the token goes outside the domain. 
The boundary condition of the model problem is associated with the game rule when the token is near the boundary.
More precisely, it is related to the area where a random walk occurs and the probability that the game ends at this point.

Lewicka and Peres \cite{MR4684373,MR4684385} invented a probabilistic interpretation for the problem \eqref{defcapf} in the case $p=2$.
The related DPP to the stochastic process coincides with \eqref{dppobp} when $\alpha=0$. They showed, among other things, that the DPP has a unique, asymptotically H\"{o}lder continuous solution up to the boundary, converging uniformly to the viscosity solution of the PDE.
We aim to extend some of their results to the range $\alpha\in (0,1)$. 
To this end, one of the main issues is how to set the game near the boundary.
We construct our game so that the token must not go outside $\Omega$,
and we still need to consider tug-of-war games in a small ball, since tug-of-war in an asymmetric domain may affect the boundary condition to the associated PDE.
For that reason, we considered tug-of-war games with shrinking step sizes as the token approaches the boundary. 
Our main result is the boundary H\"{o}lder regularity for the solution of the DPP \eqref{dppobp}, which will be restated as Theorem \ref{bdryreg}.
We also dealt with the convergence of the solution to \eqref{dppobp} as $\eps \to 0$ in Theorem \ref{convrob}.
\begin{theorem*} 
Let $\Omega \subset \mathbb{R}^{n}$ be a $C^{1,1}$-domain, $G \in C^{1}(\Gamma_{\eps})$  and $\sigma \in (0,1)$  to be determined later. 
Then, there exists a unique function $u_{\eps}$ solving \eqref{dppobp}.
Moreover, there exists $\delta_{0}\in(0,1)$ such that for every $\delta \in (0,\delta_{0})$ and $x_{0},y_{0} \in \overline{\Omega}$ 
with $|x_{0}-y_{0}|\le \delta$ and $$\dist (x_{0},\partial \Omega),\dist (y_{0},\partial \Omega)\le \delta^{1/2},$$ we have
\begin{align*}
|u_{\eps}(x_{0})-u_{\eps}(y_{0})|\le C  ||G||_{L^{\infty}(\Gamma_{\eps})}\delta^{\sigma/2} 
\end{align*}
for some $C$ depending on $n, \alpha, \gamma, \sigma$ and $\Omega$, and $\eps << \delta$. 

Moreover, there exists a function $u: \overline{\Omega} \to \mathbb{R}^{n}$ and
a subsequence $\{  \eps_{i}\}$ such that $u_{\eps_{i}}$ converges uniformly to $u$ on $\overline{\Omega}$ and
$u$ is a viscosity solution to the problem \eqref{defcapf}.
\end{theorem*}
Section \ref{sec:regularity} is generally devoted to proving the above main theorem.
The proof is technical, based on constructing suitable submartingales and carefully estimating stopping times for the game near the boundary.  
To this end, we present a geometric observation (Lemma \ref{mcs}) and an estimate of stopping times for an alternative game (Lemma \ref{esttaus}) associated with the tug-of-war setting.
Compared to the random walk case, we need to take into account the strategies of each player in the game.
Under our setting, we could develop a suitable cancellation argument (cf. \cite{MR3011990}) to deal with those controlled processes. 
And, in Section \ref{sec:oblique}, we extend our discussion to the general oblique derivative boundary value condition 
\begin{align} \label{odcon}
\langle \beta , Du \rangle  + \gamma u = \gamma G  \end{align}(Theorem \ref{bdryregob}).
 For those problems, we introduce stochastic processes containing random walks in ellipsoids relevant to the vector $\beta$. 
Several geometric observations are required to derive the desired regularity under our oblique settings 
(Proposition \ref{besob} - \ref{mvcob}).
We mention that the rotational invariance of balls plays an important role in deriving an estimate of terms associated with random walks in the normal case. 
Of course, we cannot use this property for general $\beta$ since ellipsoids are not rotationally invariant. 
We will give a proper estimate of the additional terms which arose from random walks in the ellipsoids 
so that our submartingale argument still works in the proof of Theorem \ref{bdryregob}.

The boundary value problem \eqref{defcapf} is motivated by two points of view.
When $\beta=\mathbf{n} $, the boundary condition of \eqref{defcapf}  
coincides with the Neumann ($\gamma =0$) or Robin boundary condition ($\gamma \neq0$) for the Laplace equation.
In \cite{MR3182683}, game-theoretic interpretation was considered for the Neumann boundary value condition.
Such boundary conditions have been considered for the $p$-Laplace equations in the weak theory.
One can find works related to the Neumann problems for the $\infty$-Laplacian in \cite{MR2279530,MR2426555}.
We also refer to \cite{MR2749750,MR3626320,MR4426282} for the $p$-Laplacian with the Robin boundary conditions, see also \cite{MR4459519,MR4649165,MR4691561}.
We remark that the boundary conditions are given in different forms from that of \eqref{defcapf} since the boundary value problems are derived from the Euler-Lagrange equation: For the following energy functional
$$I(u)=\int_{\Omega}|Du|^{p}\ dx+ \gamma \int_{\partial\Omega}|u|^{p}\ dS-p \gamma  \int_{\partial\Omega}gu\ dS ,$$
its corresponding Euler-Lagrange equation is
\begin{align} \label{defcapfp}
\left\{ \begin{array}{ll}
\Delta_{p}u=0 & \textrm{in $ \Omega$,}\\
|Du|^{p-2}\langle\mathbf{n}, Du\rangle  + \gamma  |u|^{p-2}u =  \gamma  G & \textrm{on $ \partial \Omega$.}\\
\end{array} \right.
\end{align} 
Meanwhile, the boundary condition can also be regarded as a special case of oblique derivative boundary conditions, that is, \eqref{odcon} with $|\langle \beta, \mathbf{n}\rangle | \ge \delta_{0} $ for some $\delta_0>0$.
We present some preceding works \cite{MR3780142,MR4223051} regarding the regularity theory for fully nonlinear oblique derivative boundary value problems.
Calderon-Zygmund type regularity estimates for these problems can be found in \cite{MR4046185,MR4283880}.
For further studies on oblique derivative boundary value problems, see for example, \cite{MR765964,MR1388600,MR1796316}.
We also refer to \cite{MR3059278} for the general theory of the boundary value problems.

{\bf Acknowledgments} This research was supported by the Basic Science Research Program through the National Research Foundation of Korea (NRF) funded by the Ministry of Education (NRF-2021R1A6A3A14045195).
The author would like to thank M. Lewicka, M. Parviainen and E. Ruosteenoja for their useful discussions.
\section{Preliminaries}
\label{sec:prelim}
We present an interpretation for our DPP \eqref{dppobp} and its related stochastic game in this section.

\subsection{Mean value characterization for \eqref{defcapf}}
\label{subsec:mvc}
We will provide a mean value characterization of solutions for the PDE \eqref{defcapf}.
This characterization gives a motivation of the DPP \eqref{dppobp}. 
We begin this subsection by introducing some notations. 
For two $n$-dimensional vectors $a=(a_{1},\dots,a_{n})$ and $b=(b_{1},\dots,b_{n})$,
we define
$$ \langle a,b \rangle = \sum_{i=1}^{n} a_{i}b_{i} .$$
Similarly, for two $n\times n$ matrices $A=(a_{ij})$ and $B=(b_{ij})$, we define
$$ \langle A:B \rangle = \sum_{i,j=1}^{n} a_{ij}b_{ij}.$$
We write $B_{r}^{k}$ to be the $k$-dimensional ball with the radius $r$, and $$B_{r,d}^{k}=B_{r}^{k}\cap \{y_{k}< d\}$$ for $0\le d\le r$.
For each $x\in \Omega$, the projection of $x$ to $\partial\Omega$ will be denoted by $\pi_{\partial\Omega}x$.
We also define \begin{align*}d_{\eps}(x) =\min \bigg\{ 1, \frac{1}{\eps} \dist(x, \partial \Omega) \bigg\} \quad \textrm{and} \quad 
d'_{\eps}(x) =\min \bigg\{ \frac{1}{2}, \frac{1}{\eps} \dist(x, \partial \Omega) \bigg\} \end{align*}
for each $x\in \overline{\Omega}$.
If there is no room for confusion, we abbreviate $d_{\eps}(x), d'_{\eps}(x)$ to $d_{x}, d'_{x}$.

Mean value characterizations for the normalized $p$-Laplacian have been studied in several previous works. 
We will set a DPP given by 
\begin{align}\label{towdpp} u_{\eps}(x)= \frac{\alpha}{2}\bigg( \sup_{B_{\frac{\eps}{2}}(x)} u_{\eps}+  
\inf_{B_{\frac{\eps}{2}}(x)} u_{\eps} \bigg) + (1-\alpha) \kint_{B_{\eps}(x)}u_{\eps}(y) dy ,\end{align}
when $x \in \Omega \backslash \Gamma_{\eps} $, where 
\begin{align}\label{estrip}\Gamma_{\eps}:=\{x\in \overline{\Omega}: \dist(x, \partial\Omega)\le \eps \}
\end{align} is the inner $\eps$-boundary strip of $\Omega$.
This DPP has a similar form to that covered in \cite{MR3011990,MR3846232}.
We remark that we take the supremum and infimum of the function within $\frac{\eps}{2}$-ball, not $\eps$-ball as in those papers, for technical reasons.

We also have to consider the boundary condition of \eqref{defcapf} to develop a mean value characterization. 
In \cite[Lemma 2.1]{MR4684373} and \cite[Lemma 2.4]{MR4684373}, 
one can find the following geometric observations, 
\begin{align}\label{bdryintav} \kint_{B_{\eps}(x)\cap\Omega}(y-x)dy = -s_{\eps}(x)\mathbf{e}_{n},
\end{align} where 
\begin{align}\label{defs}s_{\eps}(x)= \frac{|B_{1}^{n-1}|}{(n+1)|B_{1,d_{\eps}(x)}^{n}|}\eps (1 -d_{\eps}(x)^{2} )^{\frac{n+1}{2}}
\end{align} 
and
\begin{align}\label{secorder}\kint_{B_{\eps}(x)\cap\Omega}(y-x)\otimes(y-x)dy=\frac{\eps^2}{n+2}I_n + O(\eps s_{\eps}(x)),
\end{align} 
provided that $B_{\eps}(x)\cap\Omega=x+B_{\eps,\eps d_{x}}^{n} $.
Intuitively, it can be understood that if we cut an $\eps$-ball along a plane, the geometric center of the remaining part moves in the normal direction of the plane.
In this case, the factor $s_{\eps}(x)$ represents the distance from the original center of the ball.
One can also check that the integration over the cut ball does not have much effect on the equation for the Laplacian.
These geometric properties enable to derive a mean value characterization for \eqref{defcapf} with $p=2$
just by considering the integration of the function over $B_{\eps}(x)\cap\Omega$ (see \cite[Theorem 3.1]{MR4684373}).

Now we state the mean value characterization for the case $2<p<\infty$.
To this end, we need to introduce new terminology and notation. 
We say that $\Omega$ satisfies the interior ball condition with the radius $\rho$ if for each $x \in \partial\Omega$,
there exists a ball $B_{\rho}(z) \subset \Omega$ for some $z \in \Omega$ and $\rho >0$
with $x \in \partial B_{\rho}(z) $. 
We can observe that if we assume $\partial\Omega$ is a $C^{1,1}$-domain,
$\Omega$ always satisfies the interior ball condition.
Next we fix $0<r<\rho$. Then for every $y \in \Gamma_{r}$, its projection $\pi_{\partial \Omega}(y)$ to $\partial \Omega$,  is uniquely defined.
Then there is a point $z_{0} \in \Omega$ such that $ B_{\rho}(z_{0}) \subset \Omega$ and $\pi_{\partial \Omega}(y)\in \partial B_{\rho}(z_{0})  $.
In that case, we can define a function $Z^{\rho}: \Gamma_{r} \to \Omega$ to be
$Z^{\rho}(y)=z_{0} $.

One can think that it would be natural to consider the supremum or infimum over the cut ball in \eqref{towdpp} near the boundary. 
Unlike the case $p=2$, however, the corresponding boundary condition becomes different from that in \eqref{defcapf} when $p>2$
because we cannot guarantee a similar property as \eqref{bdryintav} for $\frac{1}{2}(\sup u+\inf u)$. 
To avoid this problem, we take the supremum and infimum over a smaller (but uncut) ball near the boundary.  
We will investigate the relation between \eqref{defcapf} and \eqref{dppobp} based on this observation in Section \ref{sec:converge}. 

\begin{lemma}\label{taylorpde}
Let $u \in C^{2}(\overline{\Omega})$ be a function solving the problem \eqref{defcapf} with $G \in C^{1}(\Gamma_{r_{0}})$ for some $r_{0}>0$.
Assume that $Du(x)\neq 0$ for each $x\in \overline{\Omega}$.
Then $u$ satisfies
\begin{align*} 
u(x)& = \big(1-(1-\alpha)\gamma s_{\eps}(x) \big)\bigg\{\frac{\alpha}{2}\bigg( \sup_{B_{\eps d'_{x}}(x)} u +  
\inf_{B_{\eps d'_{x}}(x)} u  \bigg) + (1-\alpha)  \kint_{B_{\eps }(x)\cap \Omega}u (y) dy \bigg\} 
\\ & \quad+ (1-\alpha)\gamma s_{\eps}(x) G(x)+O(\eps s_{\eps}(x)  )+o(\eps^{2}),
\end{align*}
where $\alpha =\frac{4(p-2)}{4p+n-6}$
for every $x \in \overline{\Omega}$ and $\eps<<r_0$.
\end{lemma}
\begin{proof} 
Fix a point $x\in \overline{\Omega}$. 
If $x \not\in \partial \Omega$, we have $B_{\eps d'_{x}}(x)\subset \Omega $. 
By using a similar argument to the proof of \cite[Theorem 2]{MR2566554}, we can check that
\begin{align}\label{inflapvar}\begin{split} \frac{1}{2}\bigg( \sup_{B_{\eps d'_{x}}(x)} u +  
\inf_{B_{\eps d'_{x}}(x)} u  \bigg) &=
u(x)+ \frac{1}{2} \frac{\langle D^{2}u(x)Du(x), Du(x) \rangle}{|Du(x)|^{2}}(\eps d'_{x})^{2}+ o(\eps^{2})
\\ & = u(x)+\frac{1}{2}\Delta_{\infty}^{N}u(x)(\eps d'_{x})^{2}+o(\eps^{2}).
\end{split}
\end{align}
Note that this observation also holds when $x\in \partial \Omega$ since both sides coincide with $u(x)$.

On the other hand, from the Taylor expansion, we observe that
\begin{align*} \kint_{B_{\eps}(x)\cap \Omega }u(y)dy&=u(x)+\bigg\langle Du(x),  \kint_{B_{\eps}(x)\cap \Omega }(y-x)dy \bigg\rangle
\\ &\quad + \frac{1}{2}\bigg\langle D^{2}u(x):\kint_{B_{\eps}(x)\cap \Omega }(y-x)\otimes(y-x)dy\bigg\rangle + o(\eps^{2}).
\end{align*}

From Lemma 2.3 in \cite{MR4684373}, we also have
\begin{align*}
\bigg\langle Du(x), \kint_{B_{\eps}(x)\cap \Omega}(y-x) dy \bigg\rangle
& = -s_{\eps}(x) \langle Du(x), \mathbf{n}(\pi_{\partial \Omega}x)\rangle + O(\eps s_{\eps}(x)  ) 
\\ & = -s_{\eps}(x) \langle Du(\pi_{\partial \Omega}x), \mathbf{n}(\pi_{\partial \Omega}x)\rangle + O(\eps s_{\eps}(x)  ) 
\\ & = \gamma s_{\eps}(x)(u(\pi_{\partial \Omega}x))-G(\pi_{\partial \Omega}x)))
+ O(\eps s_{\eps}(x)  ),
\end{align*}
and
\begin{align*}
\bigg\langle D^{2}u(x):\kint_{B_{\eps}(x)\cap \Omega }(y-x)\otimes(y-x)dy\bigg\rangle
= \frac{\Delta u(x)}{n+2}\eps^{2}+O(\eps s_{\eps}(x)).
\end{align*}
We have used the $C^{1,1}$-boundary regularity condition of $\partial \Omega$ to guarantee the well-definedness of $\mathbf{n}$.

Now we get
\begin{align}\label{add1}\begin{split}&\frac{\alpha}{2}\bigg( \sup_{B_{\eps d'_{x}}(x)} u +  
\inf_{B_{\eps d'_{x}}(x)} u  \bigg) + (1-\alpha)  \kint_{B_{\eps }(x)\cap \Omega}u (y) dy  
\\&=  u(x)+\frac{\alpha}{2}\Delta_{\infty}^{N}u(x)(d'_{x}\eps)^{2}+
(1-\alpha)\big( \gamma s_{\eps}(x)(u(\pi_{\partial \Omega}x)-G(\pi_{\partial \Omega}x))\big)
\\ &\quad +\frac{1-\alpha}{2(n+2)}\Delta u(x)\eps^{2} + O(\eps s_{\eps}(x)  )+o(\eps^{2}).
\end{split}
\end{align}

Assume that $\dist(x,\partial \Omega)\ge \eps$, that is, $d_{x}= 1$. Then we have $s_{\eps}(x)=0$ and this yields
\begin{align}\label{add2}\begin{split}&\frac{\alpha}{2}\bigg( \sup_{B_{\frac{\eps}{2}}(x)} u +  
\inf_{B_{\frac{\eps}{2}}(x)} u  \bigg) + (1-\alpha)  \kint_{B_{\eps }(x)\cap \Omega}u (y) dy  
\\&=  u(x)+\frac{\alpha}{8}\Delta_{\infty}^{N}u(x)\eps^{2}  +\frac{1-\alpha}{2(n+2)}\Delta u(x)\eps^{2} +o(\eps^{2})
\\ & = u(x)+ o(\eps^{2}).
\end{split}\end{align}
We used $\Delta_{p}^{N}u=0$ in the last equality provided
$\alpha=\frac{4(p-2)}{4p+n-6}. $

Next, we consider the case $\frac{\eps}{2}\le \dist(x,\partial \Omega)<\eps$.
We see that $d'_{x}=\frac{1}{2}$, and hence
\begin{align*}&\frac{\alpha}{2}\bigg( \sup_{B_{\eps d'_{x}}(x)} u +  
\inf_{B_{\eps d'_{x}}(x)} u  \bigg) + (1-\alpha)  \kint_{B_{\eps }(x)\cap \Omega}u (y) dy  
\\&=  u(x)+\frac{\alpha}{8}\Delta_{\infty}^{N}u(x)\eps^{2}+\frac{1-\alpha}{2(n+2)}\Delta u(x)\eps^{2} +
(1-\alpha) \gamma s_{\eps}(x)(u(\pi_{\partial \Omega}x)-G(\pi_{\partial \Omega}x)) 
\\ &\quad+ O(\eps s_{\eps}(x)  )+o(\eps^{2})
\\ &=  u(x)+ (1-\alpha) \gamma s_{\eps}(x) (u(x)-G(x))+ O(\eps s_{\eps}(x)  )+o(\eps^{2}).
\end{align*}
We used the assumptions that $u\in C^{2}(\overline{\Omega})$ and $G\in C^{1}(\Gamma_{r_{0}})$ to derive
the last estimate.
This gives 
\begin{align*}u(x)= & \frac{1}{1+(1-\alpha)\gamma s_{\eps}(x) }\bigg\{
\frac{\alpha}{2}\bigg( \sup_{B_{\eps d'_{x}}(x)} u +  
\inf_{B_{\eps d'_{x}}(x)} u  \bigg) + (1-\alpha)  \kint_{B_{\eps }(x)\cap \Omega}u (y) dy  \bigg\}
\\& +\frac{(1-\alpha)\gamma s_{\eps}(x)}{1+(1-\alpha)\gamma s_{\eps}(x)}G(x)+ O(\eps s_{\eps}(x)  )+o(\eps^{2}).
\end{align*}
Since
$$ \frac{1}{1+(1-\alpha)\gamma s_{\eps}(x) }= 1- (1-\alpha)\gamma s_{\eps}(x) + O(\eps s_{\eps}(x)),$$
we get
\begin{align}\label{add3}\begin{split}u(x)= & \big(1-(1-\alpha)\gamma s_{\eps}(x) \big)\bigg\{
\frac{\alpha}{2}\bigg( \sup_{B_{\eps d'_{x}}(x)} u +  
\inf_{B_{\eps d'_{x}}(x)} u  \bigg) + (1-\alpha)  \kint_{B_{\eps }(x)\cap \Omega}u (y) dy  \bigg\}
\\& + (1-\alpha)\gamma s_{\eps}(x)G(x)+ O(\eps s_{\eps}(x)  )+o(\eps^{2}).\end{split}
\end{align}

Finally, we take into account near the boundary case, that is, $\dist(x,\partial \Omega)< \frac{\eps}{2}$. 
Similarly, as above, we observe that
\begin{align*} &\frac{\alpha}{2}\bigg( \sup_{B_{\eps d'_{x}}(x)} u +  
\inf_{B_{\eps d'_{x}}(x)} u  \bigg) + (1-\alpha)  \kint_{B_{\eps }(x)\cap \Omega}u (y) dy  
\\&=   u(x)+\bigg(\frac{\alpha}{2}\Delta_{\infty}^{N}u(x)(d'_{x})^{2} +\frac{1-\alpha}{2(n+2)}\Delta u(x)\bigg)\eps^{2}
\\ &\quad+(1-\alpha) \gamma s_{\eps}(x)(u(\pi_{\partial \Omega}x)-G(\pi_{\partial \Omega}x)) + O(\eps s_{\eps}(x)  )
 \\ & = u(x)+ (1-\alpha) \gamma s_{\eps}(x) (u(x)-G(x)) + O(\eps s_{\eps}(x)  ).
\end{align*}
We used $$\bigg(\frac{\alpha}{2}\Delta_{\infty}^{N}u(x)(d'_{x})^{2} +\frac{1-\alpha}{2(n+2)}\Delta u(x)\bigg)\eps^{2}
=O(\eps^{2})=O(\eps s_{\eps}(x))$$
 since $u\in C^{2}(\overline{\Omega})$.
Now we have
\begin{align}\label{add4}\begin{split} 
u(x)&= \big(1-(1-\alpha)\gamma s_{\eps}(x) \big)\bigg\{\frac{\alpha}{2}\bigg( \sup_{B_{\eps d'_{x}}(x)} u +  
\inf_{B_{\eps d'_{x}}(x)} u  \bigg) + (1-\alpha)  \kint_{B_{\eps }(x)\cap \Omega}u (y) dy \bigg\} 
\\ & \quad+ (1-\alpha)\gamma s_{\eps}(x) G(x)+O(\eps s_{\eps}(x)  ).
\end{split}\end{align}

Combining \eqref{add2}, \eqref{add3} with \eqref{add4}, we complete the proof.
\end{proof}

\subsection{Setting of the stochastic game}
\label{ss:settinggame}
Here we construct a tug-of-war game associated with the DPP \eqref{dppobp}, that is,
\begin{align*}
 u_{\eps}(x)&= (1- (1-\alpha)\gamma  s_{\eps}(x)) \bigg\{\frac{\alpha}{2}\bigg( \sup_{B_{\eps d'_{x}}(x)} u_{\eps}+  
\inf_{B_{\eps d'_{x}}(x)} u_{\eps} \bigg)\\
&\quad \quad + (1-\alpha)  \kint_{B_{\eps }(x)\cap \Omega}u_{\eps}(y) dy \bigg\}  + (1-\alpha)\gamma  s_{\eps}(x)G(x).
\end{align*}
We give a brief description of our game.
Our game is basically a tug-of-war with noise, where the step size is $\frac{\eps}{2}$
and the random noise occurs in the $\eps$-ball centered at the point. 
But if the token is near the boundary,
the size of the tug-of-war is shrinking and the probability that the game ends at the point gets larger.
This game can be seen as a shrinking tug-of-war game, that is, the size of the step in the game is shrinking as the token is close to the boundary (for shrinking tug-of-war games, see \cite{MR3450753,MR4529386,MR3759468}).

Now we construct the game rigorously as follows. We begin the game at $x_{0} \in \overline{\Omega}$.
With a probability $\gamma s_{\eps}(x_{0})$, the game is over and Player II pays Player I the payoff $G(x_{0}) $.
Otherwise, the players play a tug-of-war game with noise.
They have a fair coin toss with a probability $\alpha$.
The winner of the coin toss can move the token within the ball $B_{\eps d'_{x_{0}}}(x_{0}) $.
Meanwhile, with a probability $1-\gamma s_{\eps}(x_{0})$, the token is randomly moved according to the uniform distribution in $B_{\eps}(x_{0})\cap \overline{\Omega} $.
We denote by $x_{1}$ the new position of the token.
We can define $x_{2}, x_{3}, \dots$ by repeating the same process and write $x_{\tau}$
as the point where the game is over.

Set $\tilde{C}=\{0,1\}$.
For each $k=1,2,\dots$, $\xi_{k}\in [0,1]$ is randomly selected with the uniform distribution, 
and $c \in \tilde{C}$ is defined as
\begin{align*} c_{k} =
\left\{ \begin{array}{ll}
0 & \textrm{if $ \xi_{k-1} \le 1-\gamma s_{\eps}(x)  $,}\\
1 & \textrm{if $ \xi_{k-1} > 1-\gamma s_{\eps}(x)$.}\\
\end{array} \right.
\end{align*}
A strategy is a Borel-measurable function giving the next position of the token.
For $j \in \{ I, II\}$, we define 
$$S_{j}^{k} : \to B_{ d'_{X_{k}}}(0) $$
for each $k=1,2,\dots$.
We also define $(\Upsilon,\mathcal{F},\mathbb{P}^{S_{I},S_{II}}) $,
 the countable product of $(\Upsilon_{1},\mathcal{F}_{1},\mathbb{P}_{1}^{S_{I},S_{II}}) $, by 
\begin{align*} \Upsilon&=(\Upsilon_{1})^{\mathbb{N}} \\&=\{ \omega =\{(\omega_{i},b_{i})\}_{j=1}^{\infty}
: w_{i}=\{w_{i}^{j}\}_{j=1}^{\infty}, w_{i}^{j}\in B_{1}^{n},  b_{1}\in (0,1)  
\textrm{ for all } i,j\in \mathbb{N} \}. 
\end{align*}
Furthermore, for each $k \in \mathbb{N}$, we set the probability space $(\Upsilon_{k},\mathcal{F}_{k},\mathbb{P}_{k}^{S_{I},S_{II}}) $ as the product on $k$-copies of $(\Upsilon_{1},\mathcal{F}_{1},\mathbb{P}_{1}^{S_{I},S_{II}})$.
Note that $\mathcal{F}_{k} $ is identified with the sub-$\sigma$-algebra of $\mathcal{F}$
consisting of sets $A\times \prod_{i=k+1}^{\infty}\Upsilon_{1}$ for every $A \in \mathcal{F}_{k} $.
We notice that $\{ \mathcal{F}_{k} \}_{n=0}^{\infty}$ with $\mathcal{F}_{0}=\{ \varnothing, \Upsilon \} $ is a filtration of $\mathcal{F}$. 

We define the sequence of measurable functions $\{ l_{i}^{\eps} : \Upsilon \times \overline{\Omega} \to \mathbb{N}\cup \{+\infty \} \}_{i=1}^{\infty}$ by
$$l_{i}^{\eps}(\omega, x)=\min \{ l \ge 1 : x+\eps w_{i}^{l} \in B_{\eps}(x)\cap \Omega
\} \qquad \textrm{for every } \omega \in \Upsilon, \ x \in \overline{\Omega}.$$
Since $l_{i}^{\eps} $ is $\mathbb{P}$-a.s. finite, we can also set the sequence of vector-valued random variables $\{ w_{i}^{\eps, x}\}_{i=1}$ by
$$ w_{i}^{\eps, x}(\omega)=w_{i}^{l_{i}^{\eps} (\omega, x)} \qquad  \textrm{for $\mathbb{P}$-a.e. } \omega \in \Upsilon.  $$

Now we define the sequence of vector-valued random variables $\{ X_{k}^{\eps, x_{0}}\}_{k=0}$ by
$X_{0}^{\eps, x_{0}} \equiv x_{0}$ and 
\begin{equation}\label{gadef}X_{k}^{\eps, x_{0}} =
\left\{ \begin{array}{llll}
X_{k-1}^{\eps, x_{0}}+\eps d'_{X_{k-1}}S_{I}^{k-1}& \textrm{if $0< \xi_{k-1} \le \frac{\alpha}{2}  (1-(1-\alpha)\gamma s_{\eps}(X_{k-1})) $,}\\
X_{k-1}^{\eps, x_{0}}+\eps d'_{X_{k-1}}S_{II}^{k-1} &  \\
& \hspace{-8em}\textrm{if $ (1-\frac{\alpha}{2}) (1-(1-\alpha)\gamma s_{\eps}(X_{k-1})) \le \xi_{k-1} <  1-(1-\alpha)\gamma s_{\eps}(X_{k-1})$,}\\
X_{k-1}^{\eps, x_{0}}+\eps w_{k}^{X_{k-1}^{\eps, x_{0}}} & \\
& \hspace{-9em} \textrm{if $ \frac{\alpha}{2} (1-(1-\alpha)\gamma s_{\eps}(X_{k-1})) < \xi_{k-1} < (1-\frac{\alpha}{2}) (1-(1-\alpha)\gamma s_{\eps}(X_{k-1}))$,}\\
X_{k-1}^{\eps, x_{0}} & \textrm{if $ 1-(1-\alpha)\gamma s_{\eps}(X_{k-1})< \xi_{k-1} < 1$},\\
\end{array} \right.
\end{equation}
for each $k=1,2,\dots$.
We note that $ X_{k}^{\eps, x_{0}}$ is $\mathcal{F}_{k}$-measurable for each $k \ge 1$.
Given $\gamma >0$, we define $\tau^{\eps, x_{0}}:\Upsilon \to \mathbb{N}\cup \{+\infty \}$
by
$$\tau^{\eps, x_{0}}(\omega)= \min \{ k \ge 1 : \xi_{k} < \gamma s_{\eps}(X_{k-1}^{\eps, x_{0}}) \}. $$
We also define the value functions for Players I and II as
\begin{align}\label{vf1} u_{I}^{\eps}(x_{0})=\sup_{S_{I}}\inf_{S_{II}}\mathbb{E}_{S_{I},S_{II}}^{x_{0}} [G(x_{\tau})]
\end{align}
and 
\begin{align} \label{vf2} u_{II}^{\eps}(x_{0})=\inf_{S_{II}}\sup_{S_{I}}\mathbb{E}_{S_{I},S_{II}}^{x_{0}} [G(x_{\tau})],
\end{align}
respectively.
If the above two functions coincide, we can consider the value function of the game
 as
$ u_{\eps}=u_{I}^{\eps}=u_{II}^{\eps}.$
In that case, the value function satisfies the DPP \eqref{dppobp}.

\section{The existence and uniqueness of value functions}
In this section, we deal with the existence and uniqueness of value functions.
We first prove that the operator $  T_{\eps}^{G}$ given by
\begin{align} \label{defopt}\begin{split} 
T_{\eps}^{G}u(x)= &  (1-(1-\alpha)\gamma s_{\eps}(x)) \bigg\{\frac{\alpha}{2}\bigg( \sup_{B_{\eps d'_{x}}(x)} u_{\eps}+  
\inf_{B_{\eps d'_{x}}(x)} u_{\eps} \bigg) 
\\ & \qquad \qquad  + (1-\alpha)  \kint_{B_{\eps }(x)\cap \Omega}u_{\eps}(y) dy \bigg\} + (1-\alpha)\gamma s_{\eps}(x)G(x).\end{split}
\end{align}
 preserves continuity
for each $G \in C(\Gamma)$.
We note that this guarantees the existence of measurable strategies for players in each round of our game setting.
 
\begin{lemma}
Let $T_{\eps}^{G}$ be the operator in \eqref{defopt} with $G \in C(\Gamma)$.
Then $T_{\eps}^{G}$ preserves the continuity, that is, 
$T_{\eps}^{G}u \in C(\overline{\Omega})$ for any $u \in C(\overline{\Omega})$.
\end{lemma}

\begin{proof}
First, we consider the case $x,z \in \Omega \backslash \Gamma_{\eps}$. In this case,
\begin{align*} 
 T_{\eps}^{G}u(x)&=\frac{\alpha}{2}\bigg( \sup_{B_{\frac{\eps}{2}}(x)} u_{\eps}+  
\inf_{B_{\frac{\eps}{2}}(x)} u_{\eps} \bigg) + (1-\alpha)  \kint_{B_{ }(x)\cap \Omega}u_{\eps}(y) dy  .
\end{align*}

Then we observe that
\begin{align*}
& \big| T_{\eps}^{G}u(x)- T_{\eps}^{G}u(z) \big| \\ &= \bigg|
\frac{\alpha}{2}\bigg( \sup_{B_{\frac{\eps}{2}}(x)} u+  
\inf_{B_{\frac{\eps}{2}}(x)} u -  \sup_{B_{\frac{\eps}{2}}(z)} u-
\inf_{B_{\frac{\eps}{2}}(z)} u  \bigg) \\ & \qquad+ (1-\alpha) \bigg(  \kint_{B_{\eps}(x)}u(y) dy-  \kint_{B_{\eps}(z)}u(y) dy\bigg) \bigg|
\\ & \le  \alpha   \sup_{(y,w) \in B_{\frac{\eps}{2}}(x) \times B_{\frac{\eps}{2}}(z)}|u(y)-u(w)| 
+(1-\alpha) \bigg|  \kint_{B_{\eps}(x)}u(y) dy-  \kint_{B_{\eps}(z)\cap \Omega}u(y) dy\bigg|
\\ & \le \alpha \omega_{u} (|x-z|) + 2(1-\alpha) ||u||_{L^{\infty}(\Omega)}\rho (|x-z|),
\end{align*}
where $\omega_{u}$ is the modulus of continuity of $u$ and $$\rho (t)=\left\{ \begin{array}{ll}
1-\big(1- \frac{t}{2\eps} \big)^{n} & \textrm{for $ t<2\eps$,}\\
1 & \textrm{for $ t \ge 2\eps$.}\\
\end{array} \right.$$

We also consider the case $x\in \Gamma_{\eps}$ or $z\in \Gamma_{\eps}$.
Recall that $$s_{\eps}(x)= \frac{|B_{1}^{n-1}|}{(n+1)|B_{1,d_{x}}^{n}|}\eps (1 -d_{x}^{2} )^{\frac{n+1}{2}}$$ and $0 \le d_{\eps} \le 1$.
Then we have
\begin{align*}
& |T_{\eps}^{G}u(x)- T_{\eps}^{G}u(z)|\\ &=\bigg|
( 1-   (1-\alpha)\gamma s_{\eps}(x)) \bigg\{\frac{\alpha}{2}\bigg( \sup_{B_{\eps d'_{x}}(x)} u+  
\inf_{B_{\eps d'_{x}}(x)} u \bigg) + (1-\alpha)  \kint_{B_{\eps d_{x}}(x)\cap \Omega}u(y) dy \bigg\}
\\ & \quad  - ( 1-   (1-\alpha)\gamma s_{\eps} (z) ) \bigg\{\frac{\alpha}{2}\bigg( \sup_{B_{\eps d'_{z}}(z)} u+  
\inf_{B_{\eps d'_{z}}(z)} u \bigg) + (1-\alpha)  \kint_{B_{ \eps d_{z} }(z)\cap \Omega}u(y) dy \bigg\}\bigg|
\\ & \quad + (1-\alpha)  \big|\gamma s_{\eps}(x)G(x)
 -   \gamma s_{\eps}(z)G(z)\big|
\\ & \le \omega_{s_{\eps}}(|x-z|)  ||u||_{L^{\infty}(\Omega)}+ 2\omega_{u} (|x-z|) 
\\ & \quad + 4||u||_{L^{\infty}(\Omega)}\rho (|x-z|)+\gamma  \omega_{s_{\eps}}(|x-z|)||G||_{L^{\infty}(\Gamma_{\eps})}+ \gamma\omega_{G}(|x-z|) .
\end{align*}

Since we assumed that $u$, $F$ and $s_{\eps}$ are continuous, we see that $T_{\eps}^{G}u $ is continuous when $x\in \Gamma_{\eps}$ or $z\in \Gamma_{\eps}$. This completes the proof.
\end{proof}

We also have an equiboundedness for the family of functions satisfying \eqref{dppobp}.
\begin{lemma} \label{eqbdd}
For any $0<\eps<<1$, let $u_{I}^{\eps}$ ($u_{II}^{\eps}$, respectively) be the value function for Player I (Player II, respectively) satisfying \eqref{dppobp} for a given boundary data $G \in C(\Gamma)$.
Then the family $ u_{\eps}^{G}$ is equibounded.
\end{lemma}
\begin{proof}
Recall \eqref{vf1} and \eqref{vf2}.
Since we assumed that $G \in C(\Gamma)$, we directly have
$$ ||u_{\eps}^{I}||_{L^{\infty}(\Omega)},||u_{\eps}^{II}||_{L^{\infty}(\Omega)} \le ||G||_{L^{\infty}(\Gamma_{\eps})}. $$ 
Thus, $\{ u_{I}^{\eps}\}_{\eps>0}$ and $\{ u_{II}^{\eps}\}_{\eps>0}$ are equibounded.
\end{proof}

Let $u_{0}\equiv -||G||_{L^{\infty}(\Gamma_{\eps})}$ and define $u_{k}=(T_{\eps}^{G})^{k}u_{0}$ for $n=1,2,\dots$.
We can see that $\{ u_{k}(x) \}_{k=0}$ is an increasing sequence for each $x \in \Omega$.
Then by the equiboundedness, we get that the sequence converges pointwise.
Let 
$$\underline{u}(x)=\lim_{k\to \infty} (T_{\eps}^{G})^{k}u_{0}(x). $$
Similarly, we can define
$$\overline{u}(x)=\lim_{k\to \infty} (T_{\eps}^{G})^{k}u_{0}(x) $$
for $u_{0}\equiv ||G||_{L^{\infty}(\Gamma_{\eps})}$.
We also observe that $\underline{u} \le \overline{u} $ in $\overline{\Omega}$, and
$\underline{u}$ and $ \overline{u} $ are lower and upper semicontinuous, respectively.

The following lemma gives the existence of functions satisfying \eqref{dppobp}. 
\begin{lemma}Let $\underline{u}$ and $\overline{u} $ be functions defined as above with $ G \in C(\Gamma)$. Then
these two functions satisfy \eqref{dppobp}, respectively.
\end{lemma}
\begin{proof}
Without loss of generality, it is enough to show that
\begin{align*} \underline{u}(x)&= (1- (1-\alpha)\gamma s_{\eps}(x)) \bigg\{\frac{\alpha}{2}\bigg( \sup_{B_{\eps d'_{x}}(x)}  \underline{u}+  
\inf_{B_{\eps d'_{x}}(x)} \underline{u} \bigg) + (1-\alpha)  \kint_{B_{\eps }(x)\cap \Omega} \underline{u}(y) dy \bigg\}
\\ & \qquad +  (1-\alpha)\gamma s_{\eps}(x)G(x).
\end{align*}

Since 
\begin{align*}& \underline{u}(x)=\lim_{k\to \infty}(T_{\eps}^{G}) u_{k}(x)
\\ & =(1- (1-\alpha)\gamma s_{\eps}(x))\lim_{k\to \infty} \bigg\{\frac{\alpha}{2}\bigg( \sup_{B_{\eps d'_{x}}(x)}  u_{k}+  
\inf_{B_{\eps d'_{x}}(x)} u_{k} \bigg) + (1-\alpha)  \kint_{B_{\eps }(x)\cap \Omega}\hspace{-2em}u_{k}(y) dy \bigg\}
\\ & \qquad + (1-\alpha) \gamma s_{\eps}(x)G(x)
\end{align*}
with $ u_{0}\equiv -||G||_{L^{\infty}(\Gamma_{\eps})}$, 
we need to prove that
\begin{align*} & \frac{\alpha}{2}\bigg( \sup_{B_{\eps d'_{x}}(x)}  \underline{u}+  
\inf_{B_{\eps d'_{x}}(x)} \underline{u} \bigg) + (1-\alpha)  \kint_{B_{\eps }(x)\cap \Omega} \underline{u}(y) dy
\\ & = \lim_{k\to \infty} \bigg\{\frac{\alpha}{2}\bigg( \sup_{B_{\eps d'_{x}}(x)}  u_{k}+  
\inf_{B_{\eps d'_{x}}(x)} u_{k} \bigg) + (1-\alpha)  \kint_{B_{\eps }(x)\cap \Omega} u_{k}(y) dy \bigg\}
.
\end{align*}

We first see that 
$$\sup_{B_{\eps d'_{x}}(x)} \underline{u}=  \lim_{k\to \infty} \sup_{B_{\eps d'_{x}}(x)}  u_{k},$$
since $\{  u_{k}(x)\} $ is an increasing sequence for each $x \in \overline{\Omega}$.
On the other hand, for the infimum case,   
we set $$\lambda_{0}= \lim_{k\to \infty}\inf_{B_{\eps d'_{x}}(x)} u_{k},$$ and
observe that there exists a point $y^{\ast} \in \overline{\Omega}$ such that $u_{k}(y^{\ast})\le \lambda_{0}$,
since each $u_{k}$ is continuous in $\overline{\Omega}$ and $u_{k}\le u_{k+1}$ in $\overline{\Omega}$ for all $k\ge0$. Hence,  
$$\lambda_{0}= \lim_{k\to \infty}\inf_{B_{\eps d'_{x}}(x)} u_{k}\le\inf_{B_{\eps d'_{x}}(x)} \underline{u} \le \underline{u}(y^{\ast})
= \lim_{k\to \infty}u_{k}(y^{\ast}) \le \lambda_{0} ,$$
and it gives 
$$\inf_{B_{\eps d'_{x}}(x)} \underline{u}=  \lim_{k\to \infty} \inf_{B_{\eps d'_{x}}(x)}  u_{k}. $$

Finally, since $|\underline{u}| \le  ||g||_{L^{\infty}(\Gamma_{\eps})}$, we obtain the equalities for the integral terms
by the Lebesgue dominated convergence theorem.
\end{proof}

Next, we show the uniqueness of the function satisfying \eqref{dppobp}.
To show that, we employ the arguments in preceding papers, for example, \cite[Theorem 2.3]{MR3011990} or \cite[Theorem 3.6]{MR3623556}.
\begin{lemma} \label{inin}
It holds  $\overline{u} \le u_{I}^{\eps} \le u_{II}^{\eps}\le \underline{u} $ in $\Omega$.
\end{lemma}
\begin{proof}
Since  $u_{I}^{\eps} \le u_{II}^{\eps}$, we need to show that  $u_{II}^{\eps}\le \underline{u} $ and $\overline{u} \le u_{I}^{\eps} $.
Without loss of generality, we only give proof of the first inequality.

Let $x_{0} \in \overline{\Omega}$. For $\zeta > 0$, we set a strategy $S_{II}^{*}$ for Player II such that
$$\underline{u}(S_{II}^{*}(x_{j})) =\inf_{B_{\eps d'_{x_{j}}}(x_{j})}\underline{u} + \zeta 2^{-j}.$$

Fix a strategy $S_{I}$ for Player I, and define a function $\Phi: \tilde{C} \times
\mathbb{R}^{n} \to \mathbb{R}$ such that
\begin{align*}
\Phi(c,x) = \left\{ \begin{array}{ll}
\underline{u}(x) & \textrm{if $  c=0$,}\\
G(x) & \textrm{if $ c=1$.}\\
\end{array} \right.
\end{align*}
Then we have
\begin{align*}
&\mathbb{E}_{S_{I},S_{II}^{*}}^{x_{0}}\big[ \Phi( c_{j+1},x_{j+1})+ \zeta 2^{-(j+1)}|\big(
(c_{0},x_{0}),\dots,(c_{j},x_{j}) \big) \big]
\\&= \mathbb{I}_{c_{j}}(\{0\})
(1- (1-\alpha)\gamma s_{\eps}(x_{j}))\bigg\{ \frac{\alpha}{2}   \bigg( \underline{u}(S_{I}(x_{j}))+    \underline{u} (S_{II}^{\ast}(x_{j})) \bigg)
\\ & \qquad  + (1-\alpha)  \kint_{B_{\eps}(x_{j})\cap \Omega}\underline{u}(y) dy \bigg\} 
+   \mathbb{I}_{c_{j}}(\{1\}) (1-\alpha)\gamma s_{\eps}(x_{j})G(x_{j}) + \zeta 2^{-(j+1)} \\ &
\le \mathbb{I}_{c_{j}}(\{0\})(1- (1-\alpha)\gamma s_{\eps}(x_{j}))
\bigg\{ \frac{\alpha}{2}  \bigg( \sup_{B_{\eps d'_{x_{j}}}(x_{j})} \underline{u}+  
\inf_{B_{\eps d'_{x_{j}}}(x_{j})} \underline{u}+\zeta 2^{-j} \bigg) \\ & \qquad + (1-\alpha)  \kint_{B_{\eps}(x_{j})\cap \Omega}\underline{u}(y) dy \bigg\}
 +  \mathbb{I}_{c_{j}}(\{1\}) (1-\alpha)\gamma s_{\eps}(x_{j})G(x_{j})  + \zeta 2^{-(j+1)}
\\& \le \Phi( c_{j},x_{j})+\zeta 2^{-j}.
\end{align*}
Thus, we see that $( \Phi( c_{k},x_{k})+\zeta 2^{-k})_{k=0}$ is a supermartingale.
By using the optimal stopping theorem, we have
\begin{align*}
u_{II}^{\eps}(x_{0})& = \inf_{S_{II}}\sup_{S_{I}}\mathbb{E}_{S_{I},S_{II}}^{x_{0}} [G(x_{\tau})]
\\ & \le \sup_{S_{I}}\mathbb{E}_{S_{I},S_{II}^{\ast}}^{x_{0}} [G(x_{\tau})] 
\\ & =\sup_{S_{I}}\mathbb{E}_{S_{I},S_{II}^{\ast}}^{x_{0}} [\Phi(c_{\tau+1},x_{\tau+1})] 
\\ & \le \sup_{S_{I}}\mathbb{E}_{S_{I},S_{II}^{\ast}}^{x_{0}} [\Phi(c_{\tau+1},x_{\tau+1})+\zeta 2^{-(\tau+1)}] 
\\ & \le \Phi(c_{0},x_{0})+\zeta
\\ & = \underline{u}(x_{0})+\zeta.
\end{align*}
Since $\zeta$ can be arbitrarily chosen, we have $ u_{II}^{\eps}(x_{0})\le \underline{u}(x_{0})$.

Similarly, we can also derive $  u_{I}^{\eps}(x_{0})\ge \overline{u}(x_{0})$. This completes the proof.
\end{proof}

Combining Lemma \ref{inin} with $\underline{u} \le \overline{u} $ in $\overline{\Omega}$,
we have $\overline{u} = u_{I}^{\eps}= u_{II}^{\eps}=\underline{u} =:u_{\eps}$.
Finally, we get the uniqueness and the continuity of the game value $u_{\eps}$.
\begin{corollary} Assume that $G \in C(\Gamma)$.  
There exists a unique function $u_{\eps}\in C(\overline{\Omega})$ satisfying \eqref{dppobp}.
\end{corollary}

\section{Regularity of the value function}
\label{sec:regularity}
In this section, we are concerned with the interior and boundary regularity for DPPs relevant to \eqref{dppobp}.
The main result is the boundary H\"{o}lder estimate for the function satisfying \eqref{dppobp}, Theorem \ref{bdryreg}.
For the random walk process, this type of regularity was already studied in Section 8 of \cite{MR4684385}.
Compared to that paper, the key to the proof of boundary regularity is how to control the terms arising from tug-of-war games.
Such regularity results are essential to obtain the convergence result in the next section.

We first investigate interior regularity for functions satisfying \eqref{dppobp}.
There have been several interior regularity estimates for value functions of tug-of-war games by employing coupling arguments, for example, \cite{MR3011990,MR3623556,MR4125101}.
Here we give a proof of interior estimates for the solution to \eqref{dppobp} based on the ABP and Krylov-Safanov type theory in \cite{MR4739830,MR4657420}.
 
To this end, we denote by $\mathcal{M}(B_{\Lambda})$ the set of symmetric unit Radon measures with
support in $B_{\Lambda}$.
We also set $\mu : \mathbb{R}^{n} \to \mathcal{M}(B_{\Lambda}) $ such that
$$x \mapsto \int_{B_{\Lambda}} u(x+h)d\mu_{x} (h),$$
which defines a Borel measurable function for every Borel measurable $u:  \mathbb{R}^{n} \to \mathbb{R}$.
The following notion of extremal operators can be found in \cite{MR4739830}(cf. \cite{MR4657420}).
\begin{definition}\label{exopdef}
Let $u: \mathbb{R}^{n} \to \mathbb{R}$ be a Borel measurable bounded function.
We define the extremal Pucci type operators 
\begin{align*}
\mathcal{L}_{\eps}^{+}u(x)&
:=\frac{1}{2\eps^{2}} \bigg( \alpha \sup_{\mu \in \mathcal{M}(B_{\Lambda})} \int_{B_{\Lambda}}\delta u (x, \eps h)d\mu(h)+(1-\alpha)\kint_{B_{1}}\delta u (x, \eps h)dh\bigg)
\\ & = \frac{1}{2\eps^{2}} \bigg( \alpha \sup_{h \in B_{\Lambda}} \delta u (x, \eps h)+(1-\alpha)\kint_{B_{1}}\delta u (x, \eps h)dh\bigg)
\end{align*}
and
\begin{align*}
\mathcal{L}_{\eps}^{-}u(x)&
:=\frac{1}{2\eps^{2}} \bigg( \alpha \inf_{\mu \in \mathcal{M}(B_{\Lambda})} \int_{B_{\Lambda}}\delta u (x, \eps h)d\mu(h)+(1-\alpha)\kint_{B_{1}}\delta u (x, \eps h)dh\bigg)
\\ & = \frac{1}{2\eps^{2}} \bigg( \alpha \inf_{h \in B_{\Lambda}} \delta u (x, \eps h)+(1-\alpha)\kint_{B_{1}}\delta u (x, \eps h)dh\bigg),
\end{align*}
where $\delta u(x)=u(x+\eps h)+u(x-\eps h)-2u(x) $ for every $h \in B_{\Lambda}$.
\end{definition}
We also present the following H\"{o}lder regularity result for such extremal operators.
For the proof, we refer to \cite{MR4565418,MR4496902}.
\begin{lemma}\label{hestpeo}
There exists $\eps_{0}>0$ such that if $u$ satisfies $ \mathcal{L}_{\eps}^{+}u\ge -\rho$
and  $ \mathcal{L}_{\eps}^{-}u\le \rho$ in $B_{R}$ where $\eps <\eps_{0}R$,
there exist $C>0,\sigma \in (0,1)$ such that
\begin{align*}
|u(x)-u(z)|\le \frac{C}{R^{\sigma}}\big( ||u||_{L^{\infty}(B_{R})}+R^{2}\rho \big)
(|x-z|^{\sigma}+\eps^{\sigma})
\end{align*} 
for every $x,z\in B_{R/2}$.
\end{lemma}

Now we show interior regularity for the value function $u_{\eps}$ satisfying \eqref{dppobp} based on Lemma \ref{hestpeo} and the following observation. 
Note that \eqref{dppobp} can be written as 
$$ u_{\eps}(x)=\frac{\alpha}{2}\bigg( \sup_{B_{\frac{\eps}{2}}(x)} u_{\eps}+  
\inf_{B_{\frac{\eps}{2}}(x)} u_{\eps} \bigg) + (1-\alpha)  \kint_{B_{\eps }(x)}u_{\eps}(y) dy$$
for $x\in \Omega\backslash \Gamma_{\eps}$.
For any unit vector $\nu$ in $\mathbb{R}^{n}$, we define a measure $\mu$  by 
$$\mu_{\nu}(E)=\frac{1}{2}(\delta_{E\cap \{ x+\eps\nu/2 \}} +\delta_{E\cap \{ x-\eps\nu/2 \}})$$
and 
\begin{align*}
I_{\eps}^{\nu}u_{\eps}(x)=\alpha u_{\eps}\bigg(x+\frac{\eps}{2}\nu\bigg)+ (1- \alpha)\kint_{B_{\eps}(x)}u_{\eps}(y)dy.
\end{align*}
Then, we observe
\begin{align*}
\alpha  \int_{B_{1}} u_{\eps} (x+ \eps h)d\mu_{\nu}(h)+(1-\alpha)\kint_{B_{1}} u_{\eps} (x+\eps h)dh
=\frac{1}{2}\big( I_{\eps}^{\nu}u_{\eps} (x)+  I_{\eps}^{-\nu}u_{\eps}(x)  \big)
\end{align*} 
and this implies
\begin{align*}
\alpha  \int_{B_{1}}\delta u_{\eps} (x, \eps h)d\mu_{\nu}(h)+(1-\alpha)\kint_{B_{1}} \delta u_{\eps} (x,\eps h)dh
= I_{\eps}^{\nu}u_{\eps} (x)+  I_{\eps}^{-\nu}u_{\eps}(x) -2u_{\eps}(x).
\end{align*} 

Now we get
\begin{align*}
\mathcal{L}_{\eps}^{-}u_{\eps}(x)\le
\frac{I_{\eps}^{\nu}u_{\eps} (x)+ I_{\eps}^{-\nu}u_{\eps}(x) - 2u_{\eps}(x) }{2\eps^{2}} \le 
\mathcal{L}_{\eps}^{+}u_{\eps}(x),
\end{align*}
and hence, we observe that
\begin{align*} 
0&=\frac{1}{2\eps^{2}}\bigg( \sup_{|\nu|=1} I_{\eps}^{\nu}u_{\eps} (x)+ \inf_{|\nu|=1} I_{\eps}^{\nu}u_{\eps}(x) - 2u_{\eps}(x)\bigg) 
\\ & \le  \frac{1}{2\eps^{2}} \sup_{|\nu|=1}\big( I_{\eps}^{\nu}u_{\eps} (x)+  I_{\eps}^{-\nu}u_{\eps}(x) - 2u_{\eps}(x) \big) 
\le \mathcal{L}_{\eps}^{+}u_{\eps}(x) .
\end{align*}
Similarly, we also get
$$ \mathcal{L}_{\eps}^{-}u_{\eps}(x)\le \frac{1}{2\eps^{2}} \inf_{|\nu|=1}\big(I_{\eps}^{\nu}u_{\eps} (x)+ I_{\eps}^{-\nu}u_{\eps}(x) - 2u_{\eps}(x) \big)  \le 0.$$
Therefore, we can derive the following interior H\"{o}lder estimate (cf. \cite{MR3846232,MR4125101}).
\begin{lemma}\label{intlip}Let $\Omega$ be a bounded domain and $B_{2r}(x_{0})\subset \Omega$ for some $R>0$.
Then, for the function satisfying \eqref{dppobp}, there exist $C>0$ and $\sigma\in (0,1)$ both independent of $\eps$
so that 
\begin{align} \label{intest}|u_{\eps}(x)-u_{\eps}(z)| \le C||u||_{L^{\infty}(B_{2R})}\bigg( \frac{|x-z|^{\sigma}}{R^{\sigma}}+
\frac{\eps^{\sigma}}{R^{\sigma}} \bigg) 
\end{align} for each $x,z\in B_{R}(x_{0})$. 
\end{lemma}

We state the main theorem in this section, which gives a boundary regularity estimate for \eqref{dppobp}. 
\begin{theorem}\label{bdryreg}
Let $u_{\eps}$ be the function satisfying \eqref{dppobp} with $G \in C(\Gamma_{\eps})$ for each $\eps >0$, and $\sigma \in (0,1)$ as in Lemma \ref{intlip}. 
There exists $\delta_{0}\in(0,1)$ such that for every $\delta \in (0,\delta_{0})$ and $x_{0},y_{0} \in \overline{\Omega}$ 
with $|x_{0}-y_{0}|\le \delta$ and $$\dist (x_{0},\partial \Omega),\dist (y_{0},\partial \Omega)\le \delta^{1/2},$$ we have
\begin{align*}
|u_{\eps}(x_{0})-u_{\eps}(y_{0})|\le C  ||G||_{L^{\infty}(\Gamma_{\eps})}\delta^{\sigma/2} 
\end{align*}
for some $C$ depending on $n, \alpha, \gamma, \sigma$ and $\Omega$ and $\eps << \delta$.
\end{theorem}

There have been several preceding boundary regularity results for value functions such as \cite{MR3011990,MR3494400,MR4299842}.
In these papers, constructing suitable submartingales is a key step to deriving the desired boundary estimate.
For instance, the estimate of stopping time for the game played an important role in the proof of the boundary regularity in \cite{MR3011990}. 
To this end, they constructed an alternative stochastic game and estimated its stopping time.
We want to apply a similar argument to our problem here.
Recall the definition of the function $Z^{\rho}$ in Section \ref{subsec:mvc}.
The following geometrical observation is preparatory work for that.

\begin{lemma}  \label{mcs}
Let $\Omega$ be a domain satisfying the interior ball condition with the radius $\rho>0$.
Fix $r \in (0,\frac{\rho}{2})$ and $x_{0}\in \Gamma_{r}$.
Then there exists a constant $C_{0}>0$ depending on $n, q, \rho$ and $\alpha$ such that for any sufficiently small
$\eps>0$,
\begin{align} \label{lpmo} \begin{split}
\frac{\alpha}{2}\big(|x_{0}-y_{0}|-\eps d'_{x_{0}}\big)^{-q}&\hspace{-0.5em} + \frac{\alpha}{2}\big(|x_{0}-y_{0}|+\eps d'_{x_{0}}\big)^{-q}\hspace{-0.5em} + (1-\alpha)
\kint_{B_{\eps}(x_{0})\cap \Omega}|z-y_{0}|^{-q}dz \\ & \ge |x_{0}-y_{0}|^{-q}+C_{0}(s_{\eps}(x_{0})+\eps^{2}), \end{split}
\end{align}
where $q>n-2$ and $y_{0}=Z^{\rho}(x_{0}) $. 
\end{lemma}
\begin{proof}
The basic idea of this proof comes from the convexity of the function $x \mapsto |x|^{-q}$. 
By Taylor expansion, we observe that
$$|z-y_{0}|^{-q}= |x_{0}-y_{0}|^{-q}+\langle D\phi(x_{0}), z-x_{0} \rangle
+\frac{1}{2}\langle D^{2}\phi (x_{0})(z-x_{0}),z-x_{0} \rangle + o(\eps^{2}) ,$$
where
$\phi(z)=|z-y_{0}|^{-q}$.
We also see that
$$ D\phi(x_{0})= -q|x_{0}-y_{0}|^{-q-2}(x_{0}-y_{0})$$
and
$$ D^{2}\phi(x_{0})=q|x_{0}-y_{0}|^{-q-4}\big( (q+2)(x_{0}-y_{0})\otimes
(x_{0}-y_{0})-|x_{0}-y_{0}|^{2}I_{n} \big) .$$

Then, we get
\begin{align}\label{mrge}\begin{split}
&\frac{1}{2}\big(|x_{0}-y_{0}|-\eps d'_{x_{0}}\big)^{-q} + \frac{1}{2}\big(|x_{0}-y_{0}|+\eps d'_{x_{0}}\big)^{-q} 
\\ & = |x_{0}-y_{0}|^{-q}+\frac{q(q+1)}{2}(\eps d'_{x_{0}})^{2} |x_{0}-y_{0}|^{-q-2}+o(\eps^{2})  \ge  |x_{0}-y_{0}|^{-q}\end{split}
\end{align}
and
\begin{align}\label{fote}\begin{split}
&\kint_{B_{\eps}(x_{0})\cap \Omega}|z-y_{0}|^{-q}dz
\\& = |x_{0}-y_{0}|^{-q}+\kint_{B_{\eps}(x_{0})\cap \Omega}\langle D\phi(x_{0}), z-x_{0} \rangle dz \\ & \qquad \quad
+\kint_{B_{\eps}(x_{0})\cap \Omega}\frac{1}{2}\langle D^{2}\phi (x_{0})(z-x_{0}),z-x_{0} \rangle dz + o(\eps^{2})
\\& \ge |x_{0}-y_{0}|^{-q}+ C_{1}(s_{\eps}(x_{0})+\eps^{2}),\end{split}
\end{align}
for some $C_{1}>0$ depending on $q,n$ and $ \rho$.
We remark that the last inequality is obtained by
\begin{align*}
&\kint_{B_{\eps}(x_{0})\cap \Omega}\langle D\phi(x_{0}), z-x_{0} \rangle dz
\\ & = -q|x_{0}-y_{0}|^{-q-2} \bigg\langle x_{0}-y_{0}
, \kint_{B_{\eps}(x_{0})\cap \Omega} (z-x_{0} ) dz \bigg\rangle 
\\ & = q|x_{0}-y_{0}|^{-q-1} \bigg\langle   \frac{x_{0}-y_{0}}{|x_{0}-y_{0}|}, \mathbf{n}(\pi_{\partial \Omega}(x_{0}))  \bigg\rangle s_{\eps}(x_{0}) + O(\eps s_{\eps}(x_{0}))
\ge C_{1} s_{\eps}(x_{0})
\end{align*}
and
\begin{align*}
&\kint_{B_{\eps}(x_{0})\cap \Omega} \langle D^{2}\phi (x_{0})(z-x_{0}),z-x_{0} \rangle  dz
\\ & =q|x_{0}-y_{0}|^{-q-4} \times \\ & \quad \bigg\langle (q+2)(x_{0}-y_{0})\otimes
(x_{0}-y_{0})-|x_{0}-y_{0}|^{2}I_{n} : \kint_{B_{\eps}(x_{0})\cap \Omega}  (z-x_{0})\otimes(z-x_{0} )dz\bigg\rangle
\\ & = \frac{q-n+2}{n+2}|x_{0}-y_{0}|^{2}\eps^{2} + O(\eps s_{\eps}(x_{0})) \ge C_{1}\eps^{2},
\end{align*}
provided $q>n-2$ (we can always choose such constant $C_{1}>0$). 

We finally derive \eqref{lpmo} by choosing $C_{0}= (1-\alpha)C_{1}$.
\end{proof}

Now we consider an alternative game with a slightly different setting from the original game. 
The basic setting is the same as before, but we do not finish the game even if the token is located in $\Gamma_{\eps}$.
Instead, the game must end when the token moves far enough away from the boundary.
We define the stopping time $\overline{\tau}^{\eps,\rho,h,x_{0}}$ and a sequence of random
variables $\{S_{k}^{\eps , x_{0}} \}_{k=0}^{\infty}$ by
$$\overline{\tau}^{\eps,\rho,h,x_{0}}=\min\{ k\ge 0 : |Z^{\rho}(X_{k})-X_{k}|<\rho- h \} \qquad \textrm{and} \qquad S_{k}^{\eps , x_{0}}= \sum_{j=1}^{k}s_{\eps}(X_{k}) .$$
In \cite[Lemma 8.3]{MR4684385}, we can find an estimate for $\overline{\tau}^{\eps,\rho,h,x_{0}}$ and $S_{\overline{\tau}^{\eps,\rho,h,x_{0}}}^{\eps , x_{0}}$ when $p=2.$
The following lemma is a sort of corresponding result for those stopping times under our setting.
We derive an estimate for $\overline{\tau}^{\eps,\rho,h,x_{0}}$ and $S_{\overline{\tau}^{\eps,\rho,h,x_{0}}}^{\eps , x_{0}}$ under a fixed strategy for Player II, which intuitively appears to minimize them.

\begin{lemma}\label{esttaus}Let $\overline{r}<r $ with $r$ as above.
Assume that $|x_{0}-Z^{\rho}(x_{0})  |>\rho-h$ for $h \in (0, \frac{\overline{r}}{2}-\eps)$ and we fix the strategy $S_{II}^{\ast}$ to pull towards $Z_{\rho}(X_{k})$ for Player II.
Then for every small $\eps>0$, $x_{0}\in \overline{\Omega}$, we have
\begin{align}
\sup_{S_{I}}\mathbb{E}_{S_{I},S_{II}^{\ast}}^{x_{0}}[\eps^{2}\overline{\tau}^{\eps,\rho,h,x_{0}}+S_{\overline{\tau}^{\eps,\rho,h,x_{0}}}^{\eps , x_{0}}] \le C\overline{r}^{-n-1}(h+\eps)
\end{align}
for some constant $C$ depending on $n, \alpha $ and $\rho$.
\end{lemma}
\begin{proof}
We abbreviate $ \overline{\tau}^{\eps,\rho,h,x_{0}}$ by $ \overline{\tau}$.
We consider a sequence of random variable $\{ Q_{k} \}_{k=0}$ with
$$ Q_{k}= |X_{k}-Z^{\rho}(X_{k})|^{-n}-C_{0}(k\eps^{2}+S_{k}^{\eps,x_{0}}) , $$
where $C_{0}$ is the constant in Lemma \ref{mcs}.

To show that $\{ Q_{k} \}_{k=0}^{\overline{\tau}}$ is a submartingale with the filtration $\{ \mathcal{F}_{k} \}_{k=0}$, we observe that for each $k < \overline{\tau}$,
\begin{align*}
&\sup_{S_{I}}\mathbb{E}_{S_{I},S_{II}^{\ast}}^{x_{0}}[ Q_{k+1} - Q_{k} | \mathcal{F}_{k}]
\\ & =\sup_{S_{I}}\mathbb{E}_{S_{I},S_{II}^{\ast}}^{x_{0}}[ |X_{k+1}-Z^{\rho}(X_{k+1})|^{-n} | \mathcal{F}_{k}] - |X_{k}-Z^{\rho}(X_{k})|^{-n}- C_{0}(\eps^{2}+s_{\eps}(X_{k})) .
\end{align*}
Since $\Omega$ satisfies the interior ball condition with the radius $\rho$ and $h<\frac{\overline{r}}{2}-\eps<\rho$
we have
$$ \bigg|X_{k}-\eps d'_{X_{k}}\frac{X_{k}-Z^{\rho}(X_{k})}{|X_{k}-Z^{\rho}(X_{k})|}\bigg|=|X_{k}-Z^{\rho}(X_{k})|-\eps d'_{X_{k}},$$
and 
$$ \sup_{\nu\in B_{1}}\big|X_{k}+\eps d'_{X_{k}}\nu - Z^{\rho}\big(X_{k}+d'_{X_{k}}\nu \big)\big|=|X_{k}-Z^{\rho}(X_{k})|+\eps d'_{X_{k}}.$$
Furthermore, we also see that
$$  |y-Z^{\rho}(X_{k})| \ge |y-Z^{\rho}(y)| = \rho - \dist (y, \partial \Omega)$$
for each $y \in B_{\eps}(X_{k})\cap \Omega$.
Hence, we get
\begin{align*}
&\mathbb{E}_{S_{I},S_{II}^{\ast}}^{x_{0}}[ |X_{k+1}-Z^{\rho}(X_{k+1})|^{-n} | \mathcal{F}_{k}]
\\ & \ge \frac{\alpha}{2}\big(|X_{k}-Z^{\rho}(X_{k})|-\eps d'_{X_{k}}\big)^{-n} + \frac{\alpha}{2}\big(|X_{k}-Z^{\rho}(X_{k})|+\eps d'_{X_{k}}\big)^{-n} \\ & \qquad + (1-\alpha)
\kint_{B_{\eps}(X_{k})\cap \Omega}|z-Z^{\rho}(X_{k})|^{-n}dz 
\\ & \ge  |X_{k}-Z^{\rho}(X_{k})|^{-n}+C_{0}(\eps^{2}+s_{\eps}(X_{k})) 
\end{align*}
for any strategy $S_{I}$.
From the above estimates, we see that $Q_{k}$ is a supermartingale.

Now we fix $k\ge 1$ and apply the optional stopping theorem to $\overline{\tau} \wedge k$, obtaining
\begin{align*}
\overline{r}^{-n}\le 
\mathbb{E}_{S_{I},S_{II}^{\ast}}^{x_{0}}[Q_{0}] &
\le \mathbb{E}_{S_{I},S_{II}^{\ast}}^{x_{0}}[Q_{\overline{\tau} \wedge k}] \\ &
= \mathbb{E}_{S_{I},S_{II}^{\ast}}^{x_{0}}[|X_{\overline{\tau} \wedge k}-Z^{\rho}(X_{\overline{\tau} \wedge k})|^{-n}]- C_{0} \mathbb{E}_{S_{I},S_{II}^{\ast}}^{x_{0}}[(\overline{\tau} \wedge k)\eps^{2}+S_{\overline{\tau} \wedge k}^{\eps,x_{0}}]
\\ & \le (\overline{r}-h-\eps)^{-n} - C_{0} \mathbb{E}_{S_{I},S_{II}^{\ast}}^{x_{0}}[(\overline{\tau} \wedge k)\eps^{2}+S_{\overline{\tau} \wedge k}^{\eps,x_{0}}].
\end{align*}We note that $|Q_{k}|<(\overline{r}-h-\eps)^{-n}$ and hence the optional stopping theorem is applicable.
This yields
\begin{align*}
\mathbb{E}_{S_{I},S_{II}^{\ast}}^{x_{0}}[\overline{\tau}\eps^{2}+S_{\overline{\tau} }^{\eps,x_{0}}] \le \frac{1}{C_{0}}\big(  (\overline{r}-h-\eps)^{-n} -\overline{r}^{-n} \big) \le C\overline{r}^{-n-1}(h+\eps)
\end{align*}
for any strategy $S_{I}$ and some $C$ depending on $n,\rho $ and $\alpha $.
Now we can complete the proof by taking the supremum over $S_{I}$.
\end{proof}
Now we are ready to prove the main theorem in this section.
The main difficulty compared to the case $p=2$ is to control additional terms arising from the game process.
We will employ a cancellation method similar to \cite{MR3011990}. But the original argument cannot be applied directly,
since the radius of the tug-of-war shrinks near the boundary and thus the set of $S_{I}$ and $S_{II}$ are different in this case.  
To overcome this issue, we consider a modified version of that argument.
This enables us to control additional terms due to the game setting being small enough.

From now on, we as $\eps>0$ be sufficiently small and $\overline{r}<r<\frac{\rho}{2}$.
We define $$\{ Z^{\eps, \overline{r}, x_{0}}(X_{k}),Z^{\eps, \overline{r}, y_{0}}(Y_{k}),S_{k}^{\eps, x_{0}},S_{k}^{\eps, y_{0}} \}_{n=0}$$ with the stopping times $\overline{\tau}^{\eps,\rho,h,x_{0}},\overline{\tau}^{\eps,\rho,h,y_{0}},$ respectively.
For simplicity, we use the notations $Z_{k}^{\eps, \overline{r}, x_{0}}, Z_{k}^{\eps, \overline{r}, y_{0}}$
instead of $Z^{\eps, \overline{r}, x_{0}}(X_{k}),Z^{\eps, \overline{r}, y_{0}}(Y_{k})$.
We also define
$$\overline{\tau}= \overline{\tau}^{\eps,\rho,h,x_{0}}\wedge\overline{\tau}^{\eps,\rho,h,y_{0}}.$$

\begin{proof}[Proof of Theorem \ref{bdryreg}]

The proof can be divided into three parts, similar to the proof of \cite[Theorem 8.1]{MR4684385}. We observe the following decomposition
\begin{align}\label{esdecomp}\begin{split}
&|u_{\eps}(x_{0})-u_{\eps}(y_{0}) |
\\ & \le \big|\sup_{S_{I}}\inf_{S_{II}}\mathbb{E}_{S_{I},S_{II}}^{x_{0}}[u_{\eps}(X_{\overline{\tau}})]-u_{\eps}(x_{0})\big| +\big|\sup_{S_{I}}\inf_{S_{II}}\mathbb{E}_{S_{I},S_{II}}^{y_{0}}[u_{\eps}(Y_{\overline{\tau}})]-u_{\eps}(y_{0})\big|
\\ & \quad +\big| \sup_{S_{I}}\inf_{S_{II}}\mathbb{E}_{S_{I},S_{II}}^{x_{0}}[u_{\eps}(X_{\overline{\tau}})] -\sup_{S_{I}}\inf_{S_{II}}\mathbb{E}_{S_{I},S_{II}}^{y_{0}}[u_{\eps}(Y_{\overline{\tau}})] \big|.\end{split}
\end{align}
We estimate the first two terms on the right hand side of \eqref{esdecomp} in the first step.
To deal with the last term, we will construct a proper supermartingale in the second step.
In the last step, we derive our desired estimate by using the supermartingale.

Let $h=\delta^{1/2}$ and set a sequence of random variables $\{ M_{k} \}_{k=0}$ defined as
\begin{align} \label{mgm}\begin{split} M_{k}=u_{\eps}&(X_{k})\prod_{i=1}^{k}\big( 1-  (1-\alpha)\gamma s_{\eps}(X_{i-1}) \big)\\&+ (1-\alpha)\gamma 
\sum_{j=1}^{k-1}\bigg( s_{\eps}(X_{j})G(X_{j})\prod_{i=1}^{j}\big( 1-  (1-\alpha)\gamma s_{\eps}(X_{j-1})) \bigg).
\end{split}
\end{align}
To guarantee the well-definedness of the last term on the right hand side, we considered an arbitrarily continuous extension of $G$ to $\overline{\Omega}$ in $\Omega \backslash I_{\eps}$ 
and still use the notation $G$ for convenience. 
We also remark that $M_k$ does not depend on the extension of $G$ since $s_{\eps}\equiv 0$ in $\Omega \backslash I_{\eps}$. 
Then, we observe that for each $k\ge 1$,
\begin{align*}
&\sup_{S_{I}}\inf_{S_{II}}\mathbb{E}_{S_{I},S_{II}}^{x_{0}}[M_{k+1}-M_{k}|\mathcal{F}_{k}]
\\ &= \bigg\{\frac{\alpha}{2}\bigg( \sup_{B_{\eps d'_{X_{k}}}(X_{k})}  u_{\eps}+  
\inf_{B_{\eps d'_{X_{k}}}(X_{k})} u_{\eps} \bigg) + (1-\alpha)  \kint_{B_{\eps }(X_{k})\cap \Omega} u_{\eps}(y) dy  \bigg\}\\& \qquad \qquad \times \prod_{i=1}^{k+1}\big( 1-  (1-\alpha)\gamma s_{\eps}(X_{i-1}) \big)  - u_{\eps}(X_{k})\prod_{i=1}^{k}\big( 1-  (1-\alpha)\gamma s_{\eps}(X_{i-1}) \big)
\\ & \qquad
+  (1-\alpha)\gamma  s_{\eps}(X_{k})G(X_{k})\prod_{i=1}^{k}\big( 1-  (1-\alpha)\gamma s_{\eps}(X_{j-1})\big)
\\ & = - (1-\alpha)\gamma  s_{\eps}(X_{k})G(X_{k})\prod_{i=1}^{k}\big( 1- (1-\alpha)\gamma s_{\eps}(X_{j-1})\big)
\\ & \qquad + (1-\alpha)\gamma  s_{\eps}(X_{k})G(X_{k})\prod_{i=1}^{k}\big( 1- \gamma s_{\eps}(X_{j-1})=0.
\end{align*}
Hence, $M_{k}$ is a martingale with respect to $\{ \mathcal{F}_{k} \}_{k=0}$.
Then we can apply the optional stopping theorem since
\begin{align*}
|M_{k\wedge \overline{\tau}}| \le ||G||_{L^{\infty}(\Gamma_{\eps})}+\gamma ||G||_{L^{\infty}(\Gamma_{\eps})} \sup_{S_{I}}\inf_{S_{II}}\mathbb{E}_{S_{I},S_{II}}^{x_{0}}
[S_{k\wedge \overline{\tau}-1}^{\eps, x_{0}}]
\end{align*}
is equibounded by Lemma \ref{esttaus}.
We have
\begin{align*}
&u_{\eps}(x_{0})=M_{0}\\ &=\sup_{S_{I}}\inf_{S_{II}}\mathbb{E}_{S_{I},S_{II}}^{x_{0}}[M_{ \overline{\tau}}]
\\ & =\sup_{S_{I}}\inf_{S_{II}}\mathbb{E}_{S_{I},S_{II}}^{x_{0}}\bigg[u_{\eps}(X_{\overline{\tau}})\prod_{i=1}^{\overline{\tau}}\big( 1- (1-\alpha)\gamma s_{\eps}(X_{i-1}) \big)\\ & \qquad \qquad \qquad \qquad +(1-\alpha)\gamma 
\sum_{j=1}^{\overline{\tau}-1}\bigg( s_{\eps}(X_{j})G(X_{j})\prod_{i=1}^{j}\big( 1- (1-\alpha)\gamma s_{\eps}(X_{j-1}) \bigg)\bigg].
\end{align*}
By using the following property
$$ \prod_{i=1}^{k}(1-a_{i}) \ge 1-\sum_{i=1}^{k}a_{i} $$
for every $k$ and $\{a_{i}\}_{i=1}^{k} \subset [0,1]$, we have
\begin{align*}
&u_{\eps}(x_{0})-\sup_{S_{I}}\inf_{S_{II}}\mathbb{E}_{S_{I},S_{II}}^{x_{0}}[u_{\eps}(X_{\overline{\tau}})]
\\ & = \sup_{S_{I}}\inf_{S_{II}}\mathbb{E}_{S_{I},S_{II}}^{x_{0}}\bigg[u_{\eps}(X_{\overline{\tau}})\bigg(\prod_{i=1}^{\overline{\tau}}\big( 1-(1-\alpha)\gamma s_{\eps}(X_{i-1}) \big)-1\bigg)\\ & \qquad \qquad \qquad \qquad \qquad +(1-\alpha)\gamma 
\sum_{j=1}^{\overline{\tau}-1}\bigg( s_{\eps}(X_{j})G(X_{j})\prod_{i=1}^{j}\big( 1-(1-\alpha) \gamma s_{\eps}(X_{j-1}) \bigg)
\bigg] 
\\ & \le \sup_{S_{I}}\inf_{S_{II}}\mathbb{E}_{S_{I},S_{II}}^{x_{0}}\big[ 2\gamma ||G||_{L^{\infty}(\Gamma_{\eps})}S_{\overline{\tau}}^{\eps, x_{0}} \big]
\\ &  \le  2\gamma ||G||_{L^{\infty}(\Gamma_{\eps})} \sup_{S_{I}}\mathbb{E}_{S_{I},S_{II}^{\ast}}^{x_{0}}\big[S_{\overline{\tau}}^{\eps, x_{0}} \big]
\\ &   \le  2\gamma ||G||_{L^{\infty}(\Gamma_{\eps})} \overline{r}^{-n-1}C\delta^{1/2} 
\end{align*}
for some $C=C(n, \alpha, \rho)$.
The reversed inequality can be derived by using a similar argument.
Therefore, we have 
$$|u_{\eps}(x_{0})-\sup_{S_{I}}\inf_{S_{II}}\mathbb{E}_{S_{I},S_{II}}^{x_{0}}[u_{\eps}(X_{\overline{\tau}})]| \le 2\gamma ||G||_{L^{\infty}(\Gamma_{\eps})} \overline{r}^{-n-1}C\delta^{1/2}  ,$$
$$|u_{\eps}(y_{0})-\sup_{S_{I}}\inf_{S_{II}}\mathbb{E}_{S_{I},S_{II}}^{y_{0}}[u_{\eps}(Y_{\overline{\tau}})]| \le 2\gamma ||G||_{L^{\infty}(\Gamma_{\eps})} \overline{r}^{-n-1}C\delta^{1/2}  .$$
Now we need to estimate
$$\big| \sup_{S_{I}}\inf_{S_{II}}\mathbb{E}_{S_{I},S_{II}}^{x_{0}}[u_{\eps}(X_{\overline{\tau}})] -\sup_{S_{I}}\inf_{S_{II}}\mathbb{E}_{S_{I},S_{II}}^{y_{0}}[u_{\eps}(Y_{\overline{\tau}})] \big|. $$
We write strategies for the game process with the starting point at $x_{0}$ as $S_{I}^{x},S_{II}^{x}$.
Similarly, we also use the notation $S_{I}^{y},S_{II}^{y}$ for the process starting at $y_{0}$.
Observe that
\begin{align}\label{foreas}\begin{split}
&\big| \sup_{S_{I}}\inf_{S_{II}}\mathbb{E}_{S_{I},S_{II}}^{x_{0}}[u_{\eps}(X_{\overline{\tau}})] -\sup_{S_{I}}\inf_{S_{II}}\mathbb{E}_{S_{I},S_{II}}^{y_{0}}[u_{\eps}(Y_{\overline{\tau}})] \big|
\\ &= \big| \sup_{S_{I}^{x}}\inf_{S_{II}^{x}}\inf_{S_{I}^{y}}\sup_{S_{II}^{y}}
\mathbb{E}_{S_{I}^{x},S_{II}^{x},S_{I}^{y},S_{II}^{y}}^{(x_{0},y_{0})}[u_{\eps}(X_{\overline{\tau}})-u_{\eps}(Y_{\overline{\tau}})] \big|
\\ & \le \sup_{S_{I}^{x}}\inf_{S_{II}^{x}}\inf_{S_{I}^{y}}\sup_{S_{II}^{y}}
\mathbb{E}_{S_{I}^{x},S_{II}^{x},S_{I}^{y},S_{II}^{y}}^{(x_{0},y_{0})}[|u_{\eps}(X_{\overline{\tau}})-u_{\eps}(Y_{\overline{\tau}})|] .
\end{split}\end{align}
We remark that we can consider the following decomposition:
\begin{align} \label{decomp}
\begin{split}
&\mathbb{E}_{S_{I}^{x},S_{II}^{x},S_{I}^{y},S_{II}^{y}}^{(x_{0},y_{0})}[|u_{\eps}(X_{\overline{\tau}})-u_{\eps}(Y_{\overline{\tau}})|] 
\\& = \mathbb{E}_{S_{I}^{x},S_{II}^{x},S_{I}^{y},S_{II}^{y}}^{(x_{0},y_{0})}[|u_{\eps}(X_{\overline{\tau}})-u_{\eps}(Y_{\overline{\tau}})|\mathds{1}_{\{S_{\overline{\tau}}^{\eps , x_{0}}+S_{\overline{\tau}}^{\eps , y_{0}}>1 \}}]
\\ & \ + \mathbb{E}_{S_{I}^{x},S_{II}^{x},S_{I}^{y},S_{II}^{y}}^{(x_{0},y_{0})}[|u_{\eps}(X_{\overline{\tau}})-u_{\eps}(Y_{\overline{\tau}})|\mathds{1}_{\{S_{\overline{\tau}}^{\eps , x_{0}}+S_{\overline{\tau}}^{\eps , y_{0}}\le1 \}}\mathds{1}_{\{|X_{\overline{\tau}}-Y_{\overline{\tau}}|\ge \frac{\delta_{0}}{2} \}}]
\\ & \ + \mathbb{E}_{S_{I}^{x},S_{II}^{x},S_{I}^{y},S_{II}^{y}}^{(x_{0},y_{0})}[|u_{\eps}(X_{\overline{\tau}})-u_{\eps}(Y_{\overline{\tau}})|\mathds{1}_{\{S_{\overline{\tau}}^{\eps , x_{0}}+S_{\overline{\tau}}^{\eps , y_{0}}\le1 \}}\mathds{1}_{\{|X_{\overline{\tau}}-Y_{\overline{\tau}}|< \frac{\delta_{0}}{2}  \}}].
\end{split}
\end{align}

To derive our desired result, we have to estimate $$ \sup_{S_{I}^{x}}\inf_{S_{II}^{x}}\inf_{S_{I}^{y}}\sup_{S_{II}^{y}}\mathbb{E}_{S_{I}^{x},S_{II}^{x},S_{I}^{y},S_{II}^{y}}^{(x_{0},y_{0})}[|X_{\overline{\tau}}-Y_{\overline{\tau}}|].$$
Observe that for each $k$, if $\dist(X_{k},\partial \Omega),\dist(Y_{k},\partial \Omega)\ge\frac{\eps}{2}$, we have 
\begin{align*} \sup_{S_{I}^{x}}\inf_{S_{II}^{x}}\inf_{S_{I}^{y}}\sup_{S_{II}^{y}}
\mathbb{E}_{S_{I}^{x},S_{II}^{x},S_{I}^{y},S_{II}^{y}}^{(x_{0},y_{0})}\big[|X_{k+1}-Y_{k+1}|^{2}\big|\mathcal{F}_{k}\big]
\le |X_{k}-Y_{k}|^{2},
\end{align*}
because we can copy the strategies $ S_{I}^{x}, S_{II}^{y}$ for $ S_{I}^{y}, S_{II}^{x}$ in this case.

Suppose that  $\dist(X_{k},\partial \Omega)<\frac{\eps}{2}$ or $\dist(Y_{k},\partial \Omega)<\frac{\eps}{2}$. 
For the random walk, we get the following estimate
\begin{align}\label{rwes}\begin{split}
&\mathbb{E}[|X_{k+1}-Y_{k+1}|^{2}\mathds{1}_{ \{\frac{\alpha}{2}  < \xi_{k} <  1-\frac{\alpha}{2}\}}|\mathcal{F}_{k}\}]
\\&\le (1-\alpha) |X_{k}-Y_{k}|^{2}\big(1+ C(s_{\eps}(X_{k})+s_{\eps}(Y_{k}))\big)+C\eps(s_{\eps}(X_{k})+s_{\eps}(Y_{k}))\big) \end{split}
\end{align} for some universal $C>0$
(see the proof of \cite[Theorem 8.1]{MR4684385}). 

Next, we have to handle the terms corresponding to tug-of-war games.
First, we assume that Player I won the coin toss.
In that case, we have to estimate
$$\sup_{S_{I}^{x}}\inf_{S_{I}^{y}}
\mathbb{E}_{S_{I}^{x},S_{I}^{y}}^{(x_{0},y_{0})}\big[|X_{k+1}-Y_{k+1}|^{2}\big|\mathcal{F}_{k}\big].
$$
Fix the strategy $\nu_{x}\in B_{1}$ for $S_{I}^{x}$. 
If $|X_{k}-Y_{k}|\ge\eps (d'_{X_{k}}+ d'_{Y_{k}})$, we select a unit vector $\nu_{y}$ with
$$\nu_{y}=\frac{X_{k}+\eps d'_{X_{k}}\nu_{x}-Y_{k}}{|X_{k}+\eps d'_{X_{k}}\nu_{x}-Y_{k}|},$$
which is minimizing $|X_{k+1}-Y_{k+1}|$.
Then by the observation
$$|X_{k+1}-Y_{k+1}| =\big|X_{k}+\eps d'_{X_{k}}\nu_{x}-Y_{k}\big|-\eps d'_{Y_{k}}
\le |X_{k}-Y_{k}|+ \eps (d'_{X_{k}}-d'_{Y_{k}}) ,$$
we have 
$$ \inf_{Y_{k+1}\in B_{\eps d'_{Y_{k}}}(Y_{k})}|X_{k+1}-Y_{k+1}| \le |X_{k}-Y_{k}|+\eps (d'_{X_{k}}-d'_{Y_{k}}) $$
for every possible $X_{k+1}$. Similarly, we also obtain
$$ \inf_{X_{k+1}\in B_{\eps d'_{X_{k}}}(X_{k})}|X_{k+1}-Y_{k+1}| \le |X_{k}-Y_{k}|+\eps (d'_{Y_{k}}-d'_{X_{k}}) $$
for every possible $Y_{k+1}$ if Player II won the coin toss.

Therefore, we get
\begin{align*}
& \sup_{S_{I}^{x}}\inf_{S_{II}^{x}}\inf_{S_{I}^{y}}\sup_{S_{II}^{y}}
\mathbb{E}_{S_{I}^{x},S_{II}^{x},S_{I}^{y},S_{II}^{y}}^{(x_{0},y_{0})}\big[|X_{k+1}-Y_{k+1}|^{2}\big|\mathcal{F}_{k}\big]
\\ &\le \frac{\alpha}{2}\big(  |X_{k}-Y_{k}|+\eps (d'_{X_{k}}-d'_{Y_{k}})\big)^{2}
+\frac{\alpha}{2}\big(  |X_{k}-Y_{k}|+\eps (d'_{Y_{k}}-d'_{X_{k}})\big)^{2}
\\ & \qquad + (1-\alpha) \big( |X_{k}-Y_{k}|^{2}\big(1+ C(s_{\eps}(X_{k})+s_{\eps}(Y_{k}))\big)+C\eps(s_{\eps}(X_{k})+s_{\eps}(Y_{k})))\big) 
\\ & \le |X_{k}-Y_{k}|^{2}\big(1+C(1-\alpha)(s_{\eps}(X_{k})+s_{\eps}(Y_{k}))\big)\\ &\qquad +
(1-\alpha) C\eps(s_{\eps}(X_{k})+s_{\eps}(Y_{k})))+\eps^{2} (d'_{X_{k}}-d'_{Y_{k}})^{2}.
\end{align*}
We also check that $|d'_{X_{k}}-d'_{Y_{k}}|\le \frac{1}{2}$, and this implies for any $X_{k},Y_{k} \in \Omega$ that
$$\eps^{2} (d'_{X_{k}}-d'_{Y_{k}})^{2} \le \frac{\eps^{2}}{4}\le C_{0}\eps s_{\eps}(Z)\le C_{0}\eps (s_{\eps}(X_{k})+s_{\eps}(Y_{k})), $$
where $C_{0}= \frac{(n+1)|B_{1,1/2}^{n}|}{4|B_{1}^{n-1}|}(4/3 )^{\frac{n+1}{2}}>0 $ and $Z\in \Omega$ with $d_{Z}=\frac{1}{2}$.
From the above observation, we finally get 
 \begin{align}\label{expbd}\begin{split}
& \sup_{S_{I}^{x}}\inf_{S_{II}^{x}}\inf_{S_{I}^{y}}\sup_{S_{II}^{y}}
\mathbb{E}_{S_{I}^{x},S_{II}^{x},S_{I}^{y},S_{II}^{y}}^{(x_{0},y_{0})}\big[|X_{k+1}-Y_{k+1}|^{2}\big|\mathcal{F}_{k}\big]
\\ & \le |X_{k}-Y_{k}|^{2}\big(1+C(s_{\eps}(X_{k})+s_{\eps}(Y_{k}))\big)+ C\eps(s_{\eps}(X_{k})+s_{\eps}(Y_{k})))\end{split}
\end{align}
for some $C>0$.
In the case $|X_{k}-Y_{k}|<\eps (d'_{X_{k}}+ d'_{Y_{k}})$, we see that $$|X_{k+1}-Y_{k+1}| <2\eps (d'_{X_{k}}+ d'_{Y_{k}}) \le 2\eps, $$ whatever strategies are chosen.
Hence, we have
\begin{align*}
& \sup_{S_{I}^{x}}\inf_{S_{II}^{x}}\inf_{S_{I}^{y}}\sup_{S_{II}^{y}}
\mathbb{E}_{S_{I}^{x},S_{II}^{x},S_{I}^{y},S_{II}^{y}}^{(x_{0},y_{0})}\big[|X_{k+1}-Y_{k+1}|^{2}\big|\mathcal{F}_{k}\big]
\\ &\le 4\alpha \eps^{2}+(1-\alpha)\big( \eps^{2}\big(1+ C(s_{\eps}(X_{k})+s_{\eps}(Y_{k}))\big)+C\eps(s_{\eps}(X_{k})+s_{\eps}(Y_{k}))\big)
\\ &\le C\eps^{2}+C\eps(s_{\eps}(X_{k})+s_{\eps}(Y_{k}))\big)
\\ & \le C\eps(s_{\eps}(X_{k})+s_{\eps}(Y_{k}))
\end{align*}
for some universal $C>0$.
Then we can construct the following sequence of random variables $\{ Q_{k}\}_{k=0}$ such that
$$ Q_{k}=|X_{k}-Y_{k}|^{2} e^{-C(S_{k}^{x_{0}}+S_{k}^{y_{0}})}-C\eps (S_{k}^{x_{0}}+S_{k}^{y_{0}}),$$
and check that it is a supermartingale.

Now we can apply the optional stopping theorem to $Q_{k}$ by virtue of Lemma \ref{esttaus}, since
\begin{align*}
 &\sup_{S_{I}^{x}}\inf_{S_{II}^{x}}\inf_{S_{I}^{y}}\sup_{S_{II}^{y}}\mathbb{E}_{S_{I}^{x},S_{II}^{x},S_{I}^{y},S_{II}^{y}}^{(x_{0},y_{0})}[S_{\overline{\tau}}^{\eps , x_{0}}+S_{\overline{\tau}}^{\eps , y_{0}}]
\\& \le \sup_{S_{I}^{x}}\mathbb{E}_{S_{I}^{x},S_{II}^{x,\ast}}^{x_{0}}[S_{\overline{\tau}}^{\eps , x_{0}}]+\sup_{S_{I}^{y}}\mathbb{E}_{S_{I}^{y},S_{II}^{y,\ast}}^{y_{0}}[S_{\overline{\tau}}^{\eps , y_{0}}]
\end{align*}
(we have used notations $S_{II}^{x,\ast},S_{II}^{y,\ast}$ are corresponding to $S_{II}^{\ast}$ in Lemma \ref{esttaus}).
It gives 
\begin{align*}
\delta^{2}&\ge |x_{0}-y_{0}|^{2}
\\&\ge \sup_{S_{I}^{x}}\inf_{S_{II}^{x}}\inf_{S_{I}^{y}}\sup_{S_{II}^{y}}\mathbb{E}_{S_{I}^{x},S_{II}^{x},S_{I}^{y},S_{II}^{y}}^{(x_{0},y_{0})}[|X_{\overline{\tau}}-Y_{\overline{\tau}}|^{2} e^{-C(S_{\overline{\tau}}^{\eps , x_{0}}+S_{\overline{\tau}}^{\eps , y_{0}})}-C\eps (S_{\overline{\tau}}^{\eps , x_{0}}+S_{\overline{\tau}}^{\eps , y_{0}})].
\end{align*}

From Lemma \ref{esttaus}, we already know that
\begin{align*}
\sup_{S_{I}^{x}}\sup_{S_{I}^{y}}\mathbb{E}_{S_{I}^{x},S_{II}^{x,\ast},S_{I}^{y},S_{II}^{y,\ast}}^{(x_{0},y_{0})}[S_{\overline{\tau}}^{\eps , x_{0}}+S_{\overline{\tau}}^{\eps , y_{0}}]\le C\delta^{1/2},
\end{align*}
and this yields
\begin{align}\label{proes}\mathbb{P}(S_{\overline{\tau}}^{\eps , x_{0}}+S_{\overline{\tau}}^{\eps , y_{0}}> 1)\le C\delta^{1/2}. \end{align}
We also deduce that
\begin{align*}
&\sup_{S_{I}^{x}}\inf_{S_{II}^{x}}\inf_{S_{I}^{y}}\sup_{S_{II}^{y}}\mathbb{E}_{S_{I}^{x},S_{II}^{x},S_{I}^{y},S_{II}^{y}}^{(x_{0},y_{0})}[|X_{\overline{\tau}}-Y_{\overline{\tau}}|^{2}\mathds{1}_{\{S_{\overline{\tau}}^{\eps , x_{0}}+S_{\overline{\tau}}^{\eps , y_{0}}\le1 \}}]
\\ &\le e^{-C} \sup_{S_{I}^{x}}\inf_{S_{II}^{x}}\inf_{S_{I}^{y}}\sup_{S_{II}^{y}}\mathbb{E}_{S_{I}^{x},S_{II}^{x},S_{I}^{y},S_{II}^{y}}^{(x_{0},y_{0})}[|X_{\overline{\tau}}-Y_{\overline{\tau}}|^{2} e^{-C(S_{\overline{\tau}}^{\eps , x_{0}}+S_{\overline{\tau}}^{\eps , y_{0}})}]
\\ & \le C\big( \delta^{2} + \eps\sup_{S_{I}^{x}}\sup_{S_{I}^{y}}\mathbb{E}_{S_{I}^{x},S_{II}^{x,\ast},S_{I}^{y},S_{II}^{y,\ast}}^{(x_{0},y_{0})}[S_{\overline{\tau}}^{\eps , x_{0}}+S_{\overline{\tau}}^{\eps , y_{0}}]\big)
\\ & \le C\delta^{2}+C\eps \delta^{1/2} \le C\delta^{2},
\end{align*}
 if $\eps << \delta$.
Therefore, we get 
\begin{align} \label{dtsmes}
\sup_{S_{I}^{x}}\inf_{S_{II}^{x}}\inf_{S_{I}^{y}}\sup_{S_{II}^{y}}\mathbb{E}_{S_{I}^{x},S_{II}^{x},S_{I}^{y},S_{II}^{y}}^{(x_{0},y_{0})}[|X_{\overline{\tau}}-Y_{\overline{\tau}}|^{\sigma}\mathds{1}_{\{S_{\overline{\tau}}^{\eps , x_{0}}+S_{\overline{\tau}}^{\eps , y_{0}}\le1 \}}]\le C\delta^{\sigma} .
\end{align}
Similarly, combining \eqref{proes} with \eqref{dtsmes}, we also have
\begin{align}
\mathbb{P}\bigg( \{S_{\overline{\tau}}^{\eps , x_{0}}+S_{\overline{\tau}}^{\eps , y_{0}}\le 1 \}\cap
\bigg\{ |X_{\overline{\tau}}-Y_{\overline{\tau}}| \ge \frac{\delta_{0}}{2}\bigg\} \bigg) \le \frac{C\delta}{\delta_{0}}.
\end{align}

It still remains to estimate
$$\sup_{S_{I}^{x}}\inf_{S_{II}^{x}}\inf_{S_{I}^{y}}\sup_{S_{II}^{y}}\mathbb{E}_{S_{I}^{x},S_{II}^{x},S_{I}^{y},S_{II}^{y}}^{(x_{0},y_{0})}[|u_{\eps}(X_{\overline{\tau}})-u_{\eps}(Y_{\overline{\tau}})|\mathds{1}_{\{S_{\overline{\tau}}^{\eps , x_{0}}+S_{\overline{\tau}}^{\eps , y_{0}}\le1 \}}\mathds{1}_{\{|X_{\overline{\tau}}-Y_{\overline{\tau}}|< \frac{\delta_{0}}{2}  \}}].$$
By the definition of $\overline{\tau}$, we see that
$$ \max\{\dist(X_{\overline{\tau}},\partial \Omega),\dist(Y_{\overline{\tau}},\partial \Omega)\} > \delta^{1/2}.$$
Assume that $\delta <\frac{\delta_{0}}{2}$. 
 Then, by Theorem \ref{intlip}, we have 
\begin{align*}|u_{\eps}(X_{\overline{\tau}})-u_{\eps}(Y_{\overline{\tau}}) |&\le C  ||G||_{L^{\infty}(\Gamma_{\eps})}\frac{|X_{\overline{\tau}}-Y_{\overline{\tau}}|^{\sigma} +\eps^{\sigma}}{\delta^{\sigma/2}}  
\end{align*}
a.s. in $\{S_{\overline{\tau}}^{\eps , x_{0}}+S_{\overline{\tau}}^{\eps , y_{0}}\le 1 \}\cap
\{ |X_{\overline{\tau}}-Y_{\overline{\tau}}| < \frac{\delta_{0}}{2} \} .$
Now we observe that
\begin{align*}
&\sup_{S_{I}^{x}}\inf_{S_{II}^{x}}\inf_{S_{I}^{y}}\sup_{S_{II}^{y}}\mathbb{E}_{S_{I}^{x},S_{II}^{x},S_{I}^{y},S_{II}^{y}}^{(x_{0},y_{0})}[|u_{\eps}(X_{\overline{\tau}})-u_{\eps}(Y_{\overline{\tau}})|\mathds{1}_{\{S_{\overline{\tau}}^{\eps , x_{0}}+S_{\overline{\tau}}^{\eps , y_{0}}\le1 \}}\mathds{1}_{\{|X_{\overline{\tau}}-Y_{\overline{\tau}}|< \frac{\delta_{0}}{2}\}}]
\\ & \le \sup_{S_{I}^{x}}\inf_{S_{II}^{x}}\inf_{S_{I}^{y}}\sup_{S_{II}^{y}}\frac{C  ||G||_{L^{\infty}(\Gamma_{\eps})}}{\delta^{\sigma/2}}\big( \mathbb{E}_{S_{I}^{x},S_{II}^{x},S_{I}^{y},S_{II}^{y}}^{(x_{0},y_{0})}[|X_{\overline{\tau}}-Y_{\overline{\tau}}|^{\sigma}\mathds{1}_{\{S_{\overline{\tau}}^{\eps , x_{0}}+S_{\overline{\tau}}^{\eps , y_{0}}\le1 \}}]+ \eps^{\sigma} \big)
\\ & \le \frac{C  ||G||_{L^{\infty}(\Gamma_{\eps})}(\delta^{\sigma}+\eps^{\sigma})}{\delta^{\sigma/2}} \le C  ||G||_{L^{\infty}(\Gamma_{\eps})}\delta^{\sigma/2}.
\end{align*}
We have used \eqref{dtsmes} in the second inequality.
Combining the above estimates, we finally get
\begin{align*}
&\sup_{S_{I}^{x}}\inf_{S_{II}^{x}}\inf_{S_{I}^{y}}\sup_{S_{II}^{y}}\mathbb{E}_{S_{I}^{x},S_{II}^{x},S_{I}^{y},S_{II}^{y}}^{(x_{0},y_{0})}[|u_{\eps}(X_{\overline{\tau}})-u_{\eps}(Y_{\overline{\tau}})|] \le C  ||G||_{L^{\infty}(\Gamma_{\eps})}\delta^{\sigma/2}.
\end{align*}
This implies the desired result.
\end{proof}

\section{Convergence of the value function as $\eps \to 0$}
\label{sec:converge}

Relations between a DPP and its associated equation have been investigated in many preceding studies, for example, \cite{MR2875296,MR3011990,MR3623556}.  
In this section, we take into account the convergence of the function satisfying \eqref{dppobp} as $\eps \to 0$.
We will also verify that the limit of value functions solves \eqref{defcapf} in some sense.

To this end, we first introduce a notion of viscosity solutions.

\begin{definition}[viscosity solution] \label{viscosity}
A function $u \in C(\overline{\Omega})$ is a viscosity solution to \eqref{defcapf} if
for all $x \in \overline{\Omega}$ and $\phi \in C^{2}$ such that $u(x)=\phi(x) $ and $u(y)> \phi(y) $ for $y \neq x$, we have
\begin{align*}
\left\{ \begin{array}{ll}
 \Delta_{p}^{N}\phi(x) \le 0 & \textrm{if $ x \in \Omega$,}\\
\min\big\{ \Delta_{p}^{N}\phi(x) , \gamma G(x)-\big( \langle\mathbf{n} , D\phi\rangle (x)  + \gamma \phi(x)\big) \big\} \le 0 & \textrm{if $ x \in \partial \Omega$,}\\
\end{array} \right.
\end{align*}
and 
for all $x \in \overline{\Omega}$ and $\phi \in C^{2}$ such that $u(x)=\phi(x) $ and $u(y)< \phi(y) $ for $y \neq x$, we have
\begin{align*}
\left\{ \begin{array}{ll}
 \Delta_{p}^{N}\phi (x)\ge 0 & \textrm{if $ x \in \Omega$,}\\
\max\big\{\Delta_{p}^{N}\phi(x),  \gamma G(x)-\big( \langle\mathbf{n} , D\phi\rangle(x)  + \gamma \phi(x)\big)  \big\}\ge 0   & \textrm{if $ x \in \partial \Omega$.}\\
\end{array} \right.
\end{align*}
\end{definition}

Combining Theorem \ref{bdryreg} and Theorem \ref{intlip} with an Arzel\`{a}-Ascoli criterion,
we get the following convergence result. 
%
%
\begin{theorem} \label{convrob}
Let $u_{\eps}$ be the function satisfying \eqref{dppobp} and $G \in C^{1}(\Gamma_{\eps})$ for each $\eps >0$.
Then, there exists a function $u: \overline{\Omega} \to \mathbb{R}^{n}$ and
a subsequence $\{  \eps_{i}\}$ such that $u_{\eps_{i}}$ converges uniformly to $u$ on $\overline{\Omega}$ and
$u$ is a viscosity solution to the problem \eqref{defcapf}.
\end{theorem}
\begin{proof}
For the interior case, we can prove the convergence by employing a similar argument in the proof of \cite[Theorem 4.9]{MR3011990}.

We will prove that
$$\langle\mathbf{n}  , Du\rangle  + \gamma u =\gamma G \qquad \textrm{on } \partial \Omega $$ 
in the viscosity sense.
Let $x \in \partial \Omega$ and $\phi \in C^{2}(\overline{\Omega} )$
such that $u-\phi$ has a strict local minimum at $x$.
Then we observe that
$$\inf_{B_{r}(x)}(u-\phi)= u(x)-\phi(x) \le u(z)-\phi(z) $$
for some $r >0$ any $z \in B_{r}(x)$, and the equality holds when $z=x$.

For simplicity, we just write $\{ u_{\eps} \}$ instead of $\{ u_{\eps_{i}} \}$.
Since $u_{\eps} $ converges to $u$ uniformly,  we have
\begin{align}\label{lu1}\inf_{B_{r}(x)}(u_{\eps}-\phi) < u_{\eps}(z)-\phi(z) 
\end{align}
for all $z\in B_{r}(x)\backslash \{ x\}$ when $\eps >0$ is sufficiently small.
Thus, for any $ \zeta_{\eps}>0  $, we can find a point $x_{\eps} \in B_{r}(x) \cap \overline{\Omega}$ such that
\begin{align}\label{lu2} u_{\eps}(x_{\eps})-\phi(x_{\eps}) \le  u_{\eps}(z)-\phi(z)+ \zeta_{\eps}  \end{align}
for any $z \in B_{r}(x)$ and sufficiently small $\eps >0$.
We also see that $x_{\eps} \to x$ as $\eps \to 0$.
Define $\varphi = \phi +  u_{\eps}(x_{\eps})-\phi(x_{\eps})$.
Then we have $\varphi (x_{\eps})= u_{\eps}(x_{\eps})$ and
\begin{align}\label{lu3}u_{\eps}(z) \ge u_{\eps}(x_{\eps})-\phi(x_{\eps})+\phi(z)-\zeta_{\eps} = \varphi(z)-\zeta_{\eps}
\end{align}
for each $z \in B_{r}(x)$.

Recall that $u_{\eps}$ satisfies $u_{\eps}= T_{\eps}^{G}u_{\eps} $. Hence, we get
\begin{align*}
 u_{\eps}(x_{\eps}) &=  T_{\eps}^{G}u_{\eps}(x_{\eps})
\ge T_{\eps}^{G}\varphi (x_{\eps})-(1-\gamma s_{\eps}(x_{\eps})) \zeta_{\eps}
 = T_{\eps}^{G}\phi (x_{\eps}) -\Lambda_{\eps} + \gamma s_{\eps}(x_{\eps}) \Lambda_{\eps},
\end{align*}
where $ \Lambda_{\eps}=\zeta_{\eps}+\phi (x_{\eps}) -u_{\eps}(x_{\eps}) $.
This implies
\begin{align} \label{zetaep}
\zeta_{\eps} \ge  T_{\eps}^{G}\phi (x_{\eps})- \phi (x_{\eps}) +\gamma s_{\eps}(x_{\eps}) \Lambda_{\eps}.
\end{align}

If there exists $\eps_{0}$ such that $x_{\eps} \not\in \Gamma_{\eps}$ for any $\eps<\eps_{0}$, we can show the desired result by using the argument in the interior case. Thus, it suffices to consider the case that we can choose a subsequence $\{x_{\eps_{j}}\}$ of $\{x_{\eps}\}$ satisfying $x_{\eps_{j}} \in \Gamma_{\eps_{j}}$ for each $j$.
We continue to write $x_{\eps}$ instead of $x_{\eps_{j}}$ for our convenience. 
By using Taylor expansion, we have
\begin{align*}\frac{1}{2}\phi\big(x_{\eps}+\eps d'_{x_{\eps}}\nu\big)+\frac{1}{2}\phi\big(x_{\eps}-\eps d'_{x_{\eps}}\nu \big) & = \phi(x_{\eps})+\frac{1}{2}\langle D^{2}\phi (x_{\eps})\nu, \nu\rangle  
  (d'_{x_{\eps}})^{2}\eps^{2} +o(\eps^{2}) 
  \\ & =  \phi(x_{\eps})+O(\eps^{2})
\end{align*} for  $\nu \in \partial B_{1}$,
and
\begin{align}\label{test1}\begin{split} &  \kint_{B_{\eps}(x_{\eps})\cap \Omega}\phi(y)dy  
\\ & =\phi(x)+\bigg\langle D\phi(x),  \kint_{B_{\eps}(x)\cap \Omega }(y-x)dy \bigg\rangle
\\ &\quad + \frac{1}{2}\bigg\langle D^{2}\phi(x):\kint_{B_{\eps}(x)\cap \Omega }(y-x)\otimes(y-x)dy\bigg\rangle + o(\eps^{2})
\\ &=     \phi(x_{\eps})  -s_{\eps}(x_{\eps}) \langle D\phi(x_{\eps}), \mathbf{n}(\pi_{\partial \Omega}x_{\eps})\rangle  +\frac{\Delta \phi (x_{\eps})}{n+2}\eps^{2}+ O(\eps s_{\eps}(x_{\eps})  )+o(\eps^{2})
\\& =    \phi(x_{\eps})  -s_{\eps}(x_{\eps}) \langle D\phi(x_{\eps}), \mathbf{n}(\pi_{\partial \Omega}x_{\eps})\rangle  + O(\eps s_{\eps}(x_{\eps})  ).\end{split}
\end{align} 
Hence, we see that
\begin{align} \label{test2}\begin{split}
& (1-(1-\alpha)\gamma s_{\eps}(x_{\eps}))
 \bigg\{ \frac{\alpha}{2}\phi\big(x_{\eps}+\eps d'_{x_{\eps}}\nu\big)+\frac{\alpha}{2}\phi\big(x_{\eps}-\eps d'_{x_{\eps}}\nu\big) 
 \\& \qquad \qquad + (1-\alpha) \kint_{B_{\eps}(x_{\eps})\cap \Omega}\phi(y) dy \bigg\}
+  (1-\alpha)\gamma s_{\eps}(x_{\eps})G(x_{\eps}) 
 \\ & =(1-(1-\alpha)\gamma s_{\eps}(x_{\eps}))   \big( \phi(x_{\eps})   -s_{\eps}(x_{\eps}) (1-\alpha)\langle D\phi(x_{\eps}), \mathbf{n}(\pi_{\partial \Omega}x_{\eps})\rangle \big)
 \\ & \qquad +(1-\alpha)\gamma s_{\eps}(x_{\eps})G(x_{\eps}) + O(\eps s_{\eps}(x_{\eps})  )
 \\ & = (1-(1-\alpha)\gamma s_{\eps}(x_{\eps}))   \big( \phi(x_{\eps})   -s_{\eps}(x_{\eps}) (1-\alpha)\langle D\phi(\pi_{\partial \Omega}x_{\eps}), \mathbf{n}(\pi_{\partial \Omega}x_{\eps})\rangle \big)
 \\ & \qquad +(1-\alpha)\gamma s_{\eps}(x_{\eps})G(x_{\eps}) + O(\eps s_{\eps}(x_{\eps})  ).
 \end{split}
\end{align}

Let $\nu_{1}^{\eps} \in \partial B_{1}$ satisfy
$$ \phi\big( x+ \eps d'_{x}\nu_{1}^{\eps} \big)= \inf_{B_{\eps d'_{x}}(x_{\eps})} \phi .  $$
Then, we have
\begin{align}\label{test3}\begin{split}& \frac{\alpha}{2}\sup_{B_{\eps d'_{x}}(x_{\eps})}\phi+\frac{\alpha}{2}\inf_{B_{\eps d'_{x}}(x_{\eps})}\phi + (1-\alpha) \kint_{B_{\eps}(x_{\eps})\cap \Omega}\phi(y) dy
\\ & \ge \frac{\alpha}{2}\phi\big(  x_{\eps}+\eps d'_{x_{\eps}}\nu_{1}^{\eps} \big)+\frac{\alpha}{2}\phi \big(  x_{\eps}-\eps d'_{x_{\eps}}\nu_{1}^{\eps}\big) + (1-\alpha) \kint_{B_{\eps}(x_{\eps})\cap \Omega}\phi(y) dy .
\end{split}
\end{align}
From this, we see that 
\begin{align*}  T_{\eps}^{G}\phi (x_{\eps}) & \ge (1-(1-\alpha)\gamma s_{\eps}(x_{\eps})) \bigg\{ \frac{\alpha}{2}\phi\big(  x_{\eps}+\eps d'_{x_{\eps}}\nu_{1}^{\eps} \big)+\frac{\alpha}{2}\phi \big(  x_{\eps}-\eps d'_{x_{\eps}}\nu_{1}^{\eps}\big) 
 \\ & \qquad \qquad + (1-\alpha) \kint_{B_{\eps}(x_{\eps})\cap \Omega}\phi(y) dy \bigg\}
 +  (1-\alpha)\gamma s_{\eps}(x_{\eps})G(x_{\eps}) 
 \\ & =(1-(1-\alpha)\gamma s_{\eps}(x_{\eps}))   \big( \phi(x_{\eps})   -s_{\eps}(x_{\eps}) (1-\alpha)\langle D\phi(\pi_{\partial \Omega}x_{\eps}), \mathbf{n}(\pi_{\partial \Omega}x_{\eps})\rangle \big) 
 \\ & \qquad +(1-\alpha)\gamma s_{\eps}(x_{\eps})G(x_{\eps}) + O(\eps s_{\eps}(x_{\eps})  )
\\ & =  \phi(x_{\eps})  +(1-\alpha) \big(\gamma  G(x_{\eps})- (\gamma \phi(x_{\eps})+  \langle D\phi(\pi_{\partial \Omega}x_{\eps}), \mathbf{n}(\pi_{\partial \Omega}x_{\eps})\rangle ) \big) s_{\eps} (x_{\eps}) \\ & \qquad + O(\eps s_{\eps}(x_{\eps})  ) .
\end{align*}
Combining the above estimate with \eqref{zetaep}, we get
\begin{align}\label{bdconde}\begin{split}&  \zeta_{\eps} \ge  (1-\alpha) \big(\gamma  G(x_{\eps})- (\gamma \phi(x_{\eps})+  \langle D\phi(\pi_{\partial \Omega}x_{\eps}), \mathbf{n}(\pi_{\partial \Omega}x_{\eps})\rangle ) \big)s_{\eps} (x_{\eps})\\ & \qquad + \gamma  s_{\eps}(x_{\eps}) \Lambda_{\eps} + O(\eps s_{\eps}(x_{\eps})  ).\end{split}
\end{align}
By choosing $\zeta_{\eps} = O(\eps s_{\eps}(x_{\eps})  )$, dividing by $(1-\alpha)s_{\eps}(x_{\eps})$ and taking the limit as $\eps \to 0$, we obtain
\begin{align*}0& \ge  \gamma  G(x)- \gamma  \phi(x)- \langle D\phi(x), \mathbf{n}(x)\rangle ,
\end{align*}
that is, 
$$ \langle\mathbf{n}  , D\phi\rangle (x)+\gamma  \phi(x) \ge \gamma G(x).$$
We note that we have used the $C^1$-regularity of $G$ to get the above inequalities.

Similarly, we can also obtain the reversed inequality.
Therefore, $u$ is a viscosity solution of \eqref{defcapf}.
\end{proof}

\begin{remark}
Theorem \ref{convrob} yields that value functions satisfying \eqref{dppobp} 
converges to a viscosity solution to \eqref{defcapf}.
One can also deduce H\"{o}lder regularity of the limit of the value function by using Theorem \ref{intlip} and Theorem \ref{bdryreg}.
However, these results do not imply that every solution of \eqref{defcapf} satisfies such regularity.
To guarantee this, the uniqueness of solutions to \eqref{defcapf} should be provided,
but it is still open.
\end{remark}

\section{Games for oblique derivative boundary conditions}
\label{sec:oblique}

In this section, we are concerned with games related to a generalization of the boundary condition in \eqref{defcapf}.
For a vector-valued function $\beta$, an oblique derivative boundary condition is given by 
$$\langle \beta, Du\rangle + \gamma u = G \quad \textrm{on } \partial \Omega$$
with $|\langle \beta, \mathbf{n}\rangle | \ge \delta_{0} $ on $\partial \Omega$ for given $\delta_{0}>0$.
This boundary value condition can be understood as the vector $\beta$ forming an angle bigger than a certain level with the boundary surface.
Oblique derivative boundary value problems have been studied for the past several decades.
We refer the reader to \cite{MR833695,MR923448,MR1994804} for elliptic case and \cite{MR765964,MR1111473,MR1192119} for parabolic case.
For fully nonlinear equations, one can find the existence and uniqueness of viscosity solutions in \cite{MR1104812,MR2070626}
and regularity estimates in \cite{MR2254613,MR3780142,MR4223051}.

In terms of game theory, this boundary value condition heuristically represents the situation that the random walk near the boundary occurs over an ellipsoid associated with $\beta$, instead of a ball.
We consider a stochastic process related to the following oblique derivative boundary value problem
\begin{align} \label{probge}
\left\{ \begin{array}{ll}
\Delta_{p}^{N} u=0 & \textrm{in $ \Omega$,}\\
\langle \beta,Du\rangle + \gamma  u = \gamma G & \textrm{on $ \partial \Omega$,}\\
\end{array} \right.
\end{align} 
where $\beta  =(\beta_{1},\dots, \beta_{n}): \Gamma_{\eps}  \to \mathbb{R}^{n}$ is $C^{1}$
with 
$\beta(x)=\mathbf{n}(\pi_{\partial \Omega}x)$ for any $x \in  \Gamma_{\eps}\backslash \Gamma_{\eps/2}$
and $\big|\langle \mathbf{n}(\pi_{\partial \Omega}x), \frac{\beta(x)}{|\beta(x)|} \rangle\big| \ge\zeta_{0}$ in $\Gamma_{\eps}$ for some $\zeta_{0} \in (0,1)$.
Similarly to Section 2, we investigate H\"{o}lder regularity of corresponding game values.

To construct stochastic games related to \eqref{probge},
we first need to consider ellipsoids associated with $\beta$ and investigate their geometrical properties.
And then, we show the counterparts of Lemma \ref{mcs} and \ref{esttaus} for the oblique case,  Lemma \ref{mcsob} and \ref{esttausob}.
The main theorem of this section, Theorem \ref{bdryregob}, which gives
H\"{o}lder regularity for functions satisfying \eqref{dppobpgen}, can be proved after these preparations.
The convergence of the value function will be considered in Theorem \ref{convob}.


For $x=(x_{1},\dots,x_{n} )\in \mathbb{R}^{n} $, we use the following notation
$x'= (x_{1},\dots,x_{n-1} )$ and $x=(x',x_{n})$. 

Fix a constant vector $ \overline{\beta} \in \mathbb{R}^{n}$ with $ \overline{\beta}_{n} =1$.
Let $\mathcal{T}_{\overline{\beta}}:\mathbb{R}^{n} \to \mathbb{R}^{n}$ be a linear transformation such that
$\mathcal{T}_{\overline{\beta}}\mathbf{e}_{i}=  \mathbf{e}_{i}$ for $i=1,\dots, n-1$ and 
$\mathcal{T}_{\overline{\beta}}\mathbf{e}_{n} = \overline{\beta}$.
In this case, we can observe that the unit ball $B_{1}$
is transformed by the ellipsoid
\begin{align}\label{ellipsoid}|x|^{2}-2\langle\overline{\beta} ' , x'\rangle x_{n} +  (|\overline{\beta}'|^{2}+1)x_{n}^{2}
=: \langle V_{\overline{\beta}}x,x \rangle \le 1 ,
\end{align}
where $ V_{\overline{\beta}}=(a_{ij})_{1\le i,j \le n} $ is a symmetric matrix associated with $\overline{\beta}$ with
\begin{equation}a_{ij} =
\left\{ \begin{array}{llll}
1& \textrm{if $i=j\neq n  $,}\\
|\overline{\beta}'|^{2}+1 & \textrm{if $ i=j=n$,}\\
-\overline{\beta}_{i} & \textrm{if $ i\neq n, j=n $ or $i=n, j\neq n $}\\
0 & \textrm{otherwise.}\\
\end{array} \right.
\end{equation}
We can check that $\det {A_{\overline{\beta}}}=1$.

Under the assumption $\frac{\overline{\beta}_{n}}{|\overline{\beta}|}=\frac{1}{\sqrt{|\overline{\beta}'|^{2}+1}}>\zeta_{0} $, we have 
$|\overline{\beta}'|^{2}<\frac{1}{\zeta_{0}^{2}} -1$.
We also remark that the cross-section of this ellipsoid along $x_{n}=c$ ($0\le c<1$) is a $(n-2)$-sphere.
We also set
$$ E_{r}^{ \overline{\beta} }:= \big\{ x\in \mathbb{R}^{n} :\langle V_{\overline{\beta}}x,x \rangle\le r^{2}   \big\} $$
for each $\overline{\beta}$.

Fix $y\in \Gamma_{\eps}$
and consider a rotation $P_{y}$ such that $P_{y} \mathbf{e}=\mathbf{n}(\pi_{\partial \Omega}y)$. 
For convenience, we write 
$E_{r}^{\beta, y}=P_{y}E_{r}^{\beta(y)} $.
Now we denote by
$$E_{r}^{\beta,y}(y)= y+E_{r}^{\beta,y} \qquad \textrm{and} \qquad E_{r,d}^{\beta,y}(y) =y+\big(  E_{r}^{\beta,y}   \cap \{y_{k} <d \}\big). $$
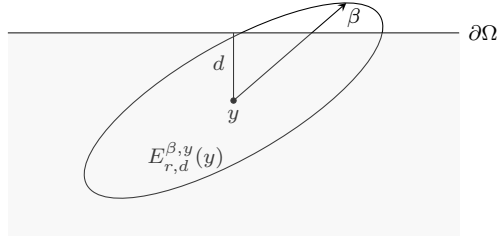
\begin{figure}
\centering
\begin{tikzpicture}

\draw[rotate=-60] (0, 0) ellipse (0.75cm and 2.25cm);

\draw (-3.0, 0.9) -- (3.0, 0.9) node[right]{\footnotesize $\partial \Omega$};
\filldraw[black]  (0,0) circle (1pt)  node [anchor=north] {\footnotesize $y$} ;
\draw [-stealth](0,0) -- (1.5, 1.299) node at (1.6,1.1) {\footnotesize $\beta$} ;
\draw (0,0) -- (0, 0.9) node at (-0.2,0.5) {\footnotesize $d$} ;
\node[above left]   at (0,  -1.1)  {\footnotesize $E_{r,d}^{\beta, y}(y)$}; 
\fill[gray!20,nearly transparent] (-3,0.9) -- (-3,-1.8) -- (3,-1.8) -- (3,0.9) -- cycle;
\end{tikzpicture}
\caption{An example of the ellipsoid $E_{r}^{\beta,y}(y)$ }
\label{fig2}
\end{figure}

Now we consider the following DPP
\begin{align} \label{dppobpgen} \begin{split}
u(x)& = \big(1-(1-\alpha)\gamma s_{\eps}(x) \big)\bigg\{\frac{\alpha}{2}\bigg( \sup_{B_{\eps d'_{x}}(x)} u +  
\inf_{B_{\eps d'_{x}}(x)} u  \bigg) + (1-\alpha)  \kint_{E_{\eps}^{ \beta ,x}(x)\cap \Omega}\hspace{-1em}u (y) dy \bigg\} 
\\ & \quad+ (1-\alpha)\gamma s_{\eps}(x) G(x).
\end{split}
\end{align}
We see that $E_{\eps}^{ \beta ,x}(x)=B_{\eps}(x)$ for all $x \in \Omega \backslash \Gamma_{\eps/2}$ by the assumption.

We define a sequence of random variables 
$\{ X_{k}^{\eps, x_{0}}\}_{k=0}$ by
$X_{0}^{\eps, x_{0}} \equiv x_{0}$ and
\begin{align*}X_{k}^{\eps, x_{0}} = X_{k-1}^{\eps, x_{0}}
+\eps w_{k}^{X_{k-1}^{\eps, x_{0}}},
\end{align*}
where $w_{k}$ is a random vector chosen in $B_{c_{0}\eps}(X_{k-1}^{\eps, x_{0}})$
with \begin{align} \label{obmove}X_{k-1}^{\eps, x_{0}}+\eps w_{k}^{X_{k-1}^{\eps, x_{0}}} \in E_{r}^{\beta,X_{k-1}^{\eps, x_{0}}}(X_{k-1}^{\eps, x_{0}})\cap \Omega \end{align}
for each $k=1,2,\dots$, where $c_{0}=c_{0}(\zeta_{0})$ satisfies
$E_{1}^{\beta,x} \subset B_{c_{0}} $
for all $x \in \Omega$. 

First, we observe the following proposition about $E^{\beta}$ for given $\beta$.
It can be derived by using \cite[Lemma 2.3]{MR4684373} since 
$\mathcal{T}_{\beta}$ is a linear transformation for any given vector $\beta$.
  This will be used to show the convergence of the value function
 as $\eps\to0$, Theorem \ref{convob}. 
\begin{proposition}\label{besob}
Then, for any $\eps << 1$, $x \in \overline{\Omega}$ and given $\beta: \Gamma_{\eps} \to \mathbb{R}^{n}$, we have the followings:
\begin{itemize}
\item[(i)] $\frac{|E_{\eps}^{\beta,x}(x)\backslash \Omega|}{|E_{\eps}^{\beta}|} \le
C\frac{s_{\eps}(x)}{\eps}$ for any $x \in \overline{\Omega}$,
\item[(ii)] $ \kint_{E_{\eps}^{\beta,x}(x)\cap \Omega}(y -x) dy = -s_{\eps}(x)\beta(x)  +O(\eps s_{\eps}(x))$ for any $x \in \overline{\Omega}$.
\end{itemize}
\end{proposition}

By Proposition \ref{besob}, we have
\begin{align*}
\bigg\langle Du(x), \kint_{E_{\eps}^{ \beta ,x}(x)\cap \Omega}(y-x) dy \bigg\rangle
& = -s_{\eps}(x) \langle Du(x), \beta(x) \rangle + O(\eps s_{\eps}(x)  ) 
\\ & = -s_{\eps}(x) \langle Du(\pi_{\partial \Omega}x), \beta(\pi_{\partial \Omega}x)\rangle + O(\eps s_{\eps}(x)  ) 
\\ & = \gamma s_{\eps}(x)(u(\pi_{\partial \Omega}x)-G(\pi_{\partial \Omega}x))
+ O(\eps s_{\eps}(x)  )
\end{align*}
provided $u\in C^{2}(\overline{\Omega})$ and $\beta \in C^{1}(\Gamma_{\eps_{0}})$.
We also see that
\begin{align*}
&\bigg\langle D^{2}u(x):\kint_{E_{\eps}^{ \beta ,x}(x)\cap \Omega }(y-x)\otimes(y-x)dy\bigg\rangle 
\\&= \frac{\langle V_{\beta(x)}^{-1} :D^{2}u(x) \rangle}{n+2}\eps^{2}+O(\eps s_{\eps}(x)) 
\\&= \frac{\langle V_{\beta(x)}^{-1}-I_{n} :D^{2}u(x) \rangle + \Delta u(x)}{n+2}\eps^{2}+O(\eps s_{\eps}(x)).
\end{align*}

Combining the above results with 
\eqref{inflapvar},
we obtain
\begin{align*}&\frac{\alpha}{2}\bigg( \sup_{B_{\eps d'_{x}}(x)} u +  
\inf_{B_{\eps d'_{x}}(x)} u  \bigg) + (1-\alpha)  \kint_{E_{\eps}^{ \beta ,x}(x)\cap \Omega}u (y) dy  
\\&=  u(x)+\frac{\alpha}{2}\Delta_{\infty}^{N}u(x)(d'_{x}\eps)^{2}+
(1-\alpha)\frac{\langle V_{\beta}^{-1}(x)-I_{n} :D^{2}u(x) \rangle + \Delta u(x)}{n+2}\eps^{2}
\\ &\quad  
+ (1-\alpha)\big( \gamma s_{\eps}(x)(u(\pi_{\partial \Omega}x)-G(\pi_{\partial \Omega}x))\big) + O(\eps s_{\eps}(x)  )+o(\eps^{2}).
\end{align*}
Repeating a similar calculation in the proof of Proposition \ref{taylorpde},
we derive the following result. 
 
\begin{proposition}\label{mvcob}
Let $u \in C^{2}(\overline{\Omega})$ be a function solving the problem \eqref{probge} with $\beta, G \in C^{1}(\Gamma_{r_{0}})$ for some $r_{0}>0$.
Assume that $Du(x)\neq 0$ for each $x\in \overline{\Omega}$.
Then we have 
\begin{align} \label{appob} \begin{split}
u(x)& = \big(1-(1-\alpha)\gamma s_{\eps}(x) \big)\bigg\{\frac{\alpha}{2}\bigg( \sup_{B_{\eps d'_{x}}(x)} u +  
\inf_{B_{\eps d'_{x}}(x)} u  \bigg) + (1-\alpha)  \kint_{E_{\eps}^{ \beta ,x}(x)\cap \Omega}\hspace{-1em}u (y) dy \bigg\} 
\\ & \quad+ (1-\alpha)\gamma s_{\eps}(x) G(x)+O(\eps s_{\eps}(x)  )+o(\eps^{2}),\end{split}
\end{align}
where $\alpha =\frac{4(p-2)}{4p+n-6}$ for every $x \in \overline{\Omega}$ and $\eps<<r_0$.
\end{proposition}

We define an operator $T_{\eps}^{\beta, G}$ to be
\begin{align} \label{appobop} \begin{split}
&T_{\eps}^{\beta, G}u(x)\\& = \big(1-(1-\alpha)\gamma s_{\eps}(x) \big)\bigg\{\frac{\alpha}{2}\bigg( \sup_{B_{\eps d'_{x}}(x)} u +  
\inf_{B_{\eps d'_{x}}(x)} u  \bigg) + (1-\alpha)  \kint_{E_{\eps}^{ \beta ,x}(x)\cap \Omega}\hspace{-1em}u (y) dy \bigg\} 
\\ & \quad+ (1-\alpha)\gamma s_{\eps}(x) G(x).\end{split}
\end{align}
Then for any $x,z \in \Gamma_{\eps/2}$, 
we have 
\begin{align*}
& |T_{\eps}^{\beta,G}u(x)- T_{\eps}^{\beta,G}u(z)|\\ &=\bigg|
( 1-  (1-\alpha)\gamma s_{\eps}(x)) \bigg\{\frac{\alpha}{2}\bigg( \sup_{B_{\eps d'_{x}}(x)} u+  
\inf_{B_{\eps d'_{x}}(x)} u \bigg) + (1-\alpha)  \kint_{E_{\eps}^{ \beta ,x}(x)\cap \Omega}u(y) dy \bigg\}
\\ & \quad  - ( 1-  (1-\alpha)\gamma s_{\eps} (z) ) \bigg\{\frac{\alpha}{2}\bigg( \sup_{B_{\eps d'_{z}}(z)} u+  
\inf_{B_{\eps d'_{z}}(z)} u \bigg) + (1-\alpha)  \kint_{E_{\eps}^{ \beta ,z}(z)\cap \Omega}u(y) dy \bigg\}\bigg|
\\ & \quad +  (1-\alpha)\big|\gamma s_{\eps}(x)G(x)
 -   \gamma s_{\eps}(z)G(z)\big|
\\ & \le \omega_{s_{\eps}}(|x-z|)  ||u||_{L^{\infty}(\Omega)}+ \omega_{u} (|x-z|) 
\\ & \quad + ||u||_{L^{\infty}(\Omega)}\rho_{\beta} (||\beta||_{L^{\infty}(\partial\Omega)}|x-z|)+\gamma  \omega_{s_{\eps}}(|x-z|)||F||_{L^{\infty}(\Gamma_{\eps})}+ \gamma\omega_{G}(|x-z|) 
\end{align*}
for some continuous function $\rho_{\beta}$ such that 
$$\frac{\big(|E_{1}^{ \beta ,x}(x)\backslash E_{1}^{ \beta ,z}(z)\big)\cup \big(|E_{\eps}^{ \beta ,z}(z)\backslash E_{\eps}^{ \beta ,x}(x)\big)|}{\eps^{n}} \le \rho_{\beta}(|x-y|).$$
Note that this is possible since $\beta$ is continuous.
Therefore, we deduce that $ T_{\eps}^{\beta, G}$ maps $C(\overline{\Omega})$ into itself.
Corresponding results to Lemma \ref{eqbdd} $-$ Lemma \ref{inin} can also be proved similarly to Section 1. Hence, we guarantee the existence and uniqueness of the value function
satisfying \eqref{dppobpgen}.
Now we can consider a stochastic game as in Section \ref{ss:settinggame}.
We set the sequence of vector-valued random variables $\{ X_{k}^{\eps, x_{0}}\}_{k=0}$ as \eqref{gadef}, but $w_k$ is randomly chosen as \eqref{obmove} when $$ \frac{\alpha}{2} (1-(1-\alpha)\gamma s_{\eps}(X_{k-1})) < \xi_{k-1} < (1-\frac{\alpha}{2}) (1-(1-\alpha)\gamma s_{\eps}(X_{k-1})).$$
The value functions for Player I and II are also samely defined as \eqref{vf1} and \eqref{vf2}.

We can find the interior regularity result for \eqref{dppobpgen} in \cite[Theorem 6.1]{MR4684385}. The following result gives a boundary H\"{o}lder estimate for the DPP.
\begin{theorem}\label{bdryregob}
Let $u_{\eps}$ be the function satisfying \eqref{dppobpgen}. 
There exists $\delta_{0}\in(0,1)$ such that for every $\delta \in (0,\delta_{0})$ and $x_{0},y_{0} \in \overline{\Omega}$ 
with $|x_{0}-y_{0}|\le \delta$ and $$\dist (x_{0},\partial \Omega),\dist (y_{0},\partial \Omega)\le \delta^{1/2}.$$ Then, for some $\sigma \in (0,1)$ in Lemma \ref{intlip}, we have
\begin{align*}
|u_{\eps}(x_{0})-u_{\eps}(y_{0})|\le  ||G||_{L^{\infty}(\Gamma_{\eps})}C\delta^{\sigma/2} 
\end{align*}
for some $C$ depending on $n, \alpha, \gamma, \sigma,  \zeta_{0}$ and $\Omega$ and $\eps << \delta$.
\end{theorem}
\begin{proof}
We follow the proof of Theorem \ref{bdryreg}.
We maintain the notations $Z_{k}^{\eps, \overline{r}, x_{0}}, Z_{k}^{\eps, \overline{r}, y_{0}}$ and $\overline{\tau}$ in Section 2.
 
We set   
\begin{align*} M_{k}=u_{\eps}(X_{k})&\prod_{i=1}^{k}\big( 1- (1-\alpha)\gamma s_{\eps}(X_{i-1}) \big)
\\&+(1-\alpha)\gamma 
\sum_{j=1}^{k-1}\bigg( s_{\eps}(X_{j})G(X_{j})\prod_{i=1}^{j}\big( 1- (1-\alpha)\gamma s_{\eps}(X_{j-1})) \bigg)
\end{align*}
and check that $M_{k}$ is a martingale from a similar observation in the previous section
(we use again a Lipschitz extension of $G$ to $\Omega$ with $G\equiv0$ outside $\Gamma_{\eps}$).
We also obtain
\begin{align*}
&u_{\eps}(x_{0})-\sup_{S_{I}}\inf_{S_{II}}\mathbb{E}_{S_{I},S_{II}}^{x_{0}}[u_{\eps}(X_{\overline{\tau}})]
  \le  2\gamma ||G||_{L^{\infty}(\Gamma_{\eps})} \overline{r}^{-n-1}C\delta^{1/2} ,
\end{align*} 
\begin{align*}
&u_{\eps}(y_{0})-\sup_{S_{I}}\inf_{S_{II}}\mathbb{E}_{S_{I},S_{II}}^{y_{0}}[u_{\eps}(Y_{\overline{\tau}})]
  \le  2\gamma ||G||_{L^{\infty}(\Gamma_{\eps})} \overline{r}^{-n-1}C\delta^{1/2} 
\end{align*} 
by using a similar calculation as before.

Again, we deal with 
$$ \sup_{S_{I}^{x}}\inf_{S_{II}^{x}}\inf_{S_{I}^{y}}\sup_{S_{II}^{y}}
\mathbb{E}_{S_{I}^{x},S_{II}^{x},S_{I}^{y},S_{II}^{y}}^{(x_{0},y_{0})}[|u_{\eps}(X_{\overline{\tau}})-u_{\eps}(Y_{\overline{\tau}})|] $$
with the decomposition \eqref{decomp}
to get an estimate for
$$\big|\sup_{S_{I}}\inf_{S_{II}}\mathbb{E}_{S_{I},S_{II}}^{x_{0}}[u_{\eps}(X_{\overline{\tau}})] -\sup_{S_{I}}\inf_{S_{II}}\mathbb{E}_{S_{I},S_{II}}^{y_{0}}[u_{\eps}(Y_{\overline{\tau}})] \big|. $$

To this end, we first estimate $$\sup_{S_{I}^{x}}\inf_{S_{II}^{x}}\inf_{S_{I}^{y}}\sup_{S_{II}^{y}}\mathbb{E}_{S_{I}^{x},S_{II}^{x},S_{I}^{y},S_{II}^{y}}^{(x_{0},y_{0})}[|X_{\overline{\tau}}-Y_{\overline{\tau}}|].$$
We only consider the case  $$\min\{\dist(X_{k},\partial \Omega),\dist(Y_{k},\partial \Omega)\}<\eps,$$ since the other case was already covered in the proof of Theorem \ref{bdryreg}. 
We aim to show the inequality \eqref{expbd}
for some universal $C>0$. 
By a similar argument in the proof of Theorem \ref{bdryreg}, we can get the following estimate
\begin{align*}
& \sup_{S_{I}^{x}}\inf_{S_{II}^{x}}\inf_{S_{I}^{y}}\sup_{S_{II}^{y}}
\mathbb{E}_{S_{I}^{x},S_{II}^{x},S_{I}^{y},S_{II}^{y}}^{(x_{0},y_{0})}\big[|X_{k+1}-Y_{k+1}|^{2}\big|\mathcal{F}_{k}\big]
\\ &\le \frac{\alpha}{2}\big(  |X_{k}-Y_{k}|+\eps (d'_{X_{k}}-d'_{Y_{k}})\big)^{2}
+\frac{\alpha}{2}\big(  |X_{k}-Y_{k}|+\eps (d'_{Y_{k}}-d'_{X_{k}})\big)^{2}
\\ & \quad + \mathbb{E}[|X_{k+1}-Y_{k+1}|^{2}\mathds{1}_{ \{\frac{\alpha}{2}  < \xi_{k} <  1-\frac{\alpha}{2}\}}|\mathcal{F}_{k}\}].
\end{align*}
To estimate the last term in the above inequality, we first observe that
\begin{align}\label{mgbd}\begin{split}
\mathbb{E}[|X_{k+1}-Y_{k+1}|^{2}|\mathcal{F}_{k}]&=
\mathbb{E}[|(X_{k+1}-Y_{k+1})-(X_{k}-Y_{k})|^{2}|\mathcal{F}_{k}]-|X_{k}-Y_{k}|^{2}
\\ & \qquad   + 2 \langle \mathbb{E}[X_{k+1}-Y_{k+1}|\mathcal{F}_{k}],X_{k}-Y_{k} \rangle.
\end{split}
\end{align} 
Through a geometric observation in \cite[Lemma 7.1]{MR4684385}, we get the following estimate: 
\begin{align*}
&  2 \langle \mathbb{E}[X_{k+1}-Y_{k+1}|\mathcal{F}_{k}],X_{k}-Y_{k} \rangle-|X_{k}-Y_{k}|^{2}
\\& \le |X_{k}-Y_{k}|^{2} + C(|X_{k}-Y_{k}|^{2}+\eps)(s_{\eps}(X_{k})+s_{\eps}(Y_{k})).
\end{align*}
Next, we have to obtain an estimate for 
$$\mathbb{E}[|(X_{k+1}-Y_{k+1})-(X_{k}-Y_{k})|^{2}|\mathcal{F}_{k}] .$$
If $\min\{\dist(X_{k}(\omega),\Omega),\dist(Y_{k}(\omega),\Omega)\}\ge\eps$, we have $$l_{k+1}^{\eps}(\omega,X_{k})= l_{k+1}^{\eps}(\omega,Y_{k}) $$ and thus the term can be cancelled.
When $\min\{\dist(X_{k}(\omega),\Omega),\dist(Y_{k}(\omega))\}\ge\eps/2$,
we see that
\begin{align*}
\mathbb{E}[|(X_{k+1}-Y_{k+1})-(X_{k}-Y_{k})|^{2}|\mathcal{F}_{k}]&
\le \mathbb{E}[|2\eps \cdot \mathds{1}_{\{l_{k+1}^{\eps}(\omega,X_{k})\neq l_{k+1}^{\eps}(\omega,Y_{k})\}}|^{2}|\mathcal{F}_{k}]
\\ & \le C\eps^{2} \bigg(\frac{s_{\eps}(X_{k})}{\eps}+\frac{s_{\eps}(Y_{k})}{\eps}\bigg)
\\ & = C\eps (s_{\eps}(X_{k})+s_{\eps}(Y_{k})),
\end{align*}
because we assumed $\beta(x)=\mathbf{n}(\pi_{\partial \Omega}x)$ for any $x \in  \Gamma_{\eps}\backslash \Gamma_{\eps/2}$.
Now we suppose that $\min\{\dist(X_{k}(\omega),\Omega),\dist(Y_{k}(\omega))\}<\eps/2$.
Without loss of generality, we assume that $ \dist(X_{k}(\omega),\Omega)<\eps/2$.
Then, we have
\begin{align*}
\mathbb{E}[|(X_{k+1}-Y_{k+1})-(X_{k}-Y_{k})|^{2}|\mathcal{F}_{k}]&
\le \mathbb{E}[|C(\zeta_{0})\eps \cdot \mathds{1}_{\{l_{k+1}^{\eps}(\omega,X_{k})\neq l_{k+1}^{\eps}(\omega,Y_{k})\}}|^{2}|\mathcal{F}_{k}]
\\ & \le C(\zeta_{0})\eps^{2} \le C(\zeta_{0},n)\eps(s_{\eps}(X_{k})+s_{\eps}(Y_{k})),
\end{align*}
since $ s_{\eps}(X_{k})\ge C(n)\eps$ for some universal $C(n)>0$.
Consequently,
\begin{align*}
\mathbb{E}[|(X_{k+1}-Y_{k+1})-(X_{k}-Y_{k})|^{2}|\mathcal{F}_{k}]
\le C \eps(s_{\eps}(X_{k})+s_{\eps}(Y_{k}))
\end{align*}
for some universal $C>0$, and this implies \eqref{expbd}.

From the above observation, we construct a sequence of random variables $\{ Q_{k}\}_{k=0}$ to be
$$ Q_{k}=|X_{k}-Y_{k}|^{2} e^{-C(S_{k}^{x_{0}}+S_{k}^{y_{0}})}-C\eps (S_{k}^{x_{0}}+S_{k}^{y_{0}}),$$
and we can show that $Q_{k}$ is a supermartingale.
The rest of the proof runs as in the proof of Theorem \ref{bdryreg}, by Lemma \ref{mcsob} and \ref{esttausob}.
\end{proof}
 
The following lemmas are counterparts of Lemma \ref{mcs} and Lemma \ref{esttaus}, which are needed to get the boundary estimates.
 
\begin{lemma}  \label{mcsob}
Let $\Omega$ be a domain satisfying the interior ball condition with the radius $\rho>0$
and $\beta \in C^{1}(\Gamma_{\rho})$.
Fix $r \in (0,\frac{\rho}{2})$ and $x_{0}\in \Gamma_{r}$.
Then there exists a constant $C_{0}>0$ depending on $n, q, \rho,\alpha$ and $\zeta_{0}$ such that for any small
$\eps>0$,
\begin{align} \label{lpmo} \begin{split}
\frac{\alpha}{2}\big(|x_{0}-y_{0}|-\eps d'_{x_{0}} \big)^{-q}&\hspace{-0.5em} + \frac{\alpha}{2}\big(|x_{0}-y_{0}|+\eps d'_{x_{0}} \big)^{-q}\hspace{-0.5em} + (1-\alpha)
\kint_{E_{\eps}^{ \beta ,x_{0}}(x_{0})\cap \Omega}\hspace{-1em}|z-y_{0}|^{-q}dz \\ & \ge |x_{0}-y_{0}|^{-q}+C_{0}(s_{\eps}(x_{0})+\eps^{2}), \end{split}
\end{align}
where $q>(\zeta_{0}^{-2}+n-1 )(\zeta_{0}^{-2}+1 )-2$ and $y_{0}=Z^{\rho}(x_{0}) $. 
\end{lemma}
\begin{proof}
Since we already have \eqref{mrge}, it is sufficient to estimate 
\begin{align}
\kint_{E_{\eps}^{ \beta ,x_{0}}(x_{0})\cap \Omega}\langle D\phi(x_{0}), z-x_{0} \rangle dz
\end{align}
for the function $\phi(z)=|z-y_{0}|^{-q}$.
Observe that
\begin{align}\label{comb2}\begin{split}
&\kint_{E_{\eps}^{ \beta ,x_{0}}(x_{0})\cap \Omega}\langle D\phi(x_{0}), z-x_{0} \rangle dz
\\ & = -q|x_{0}-y_{0}|^{-q-2} \bigg\langle x_{0}-y_{0}
, \kint_{E_{\eps}^{ \beta ,x_{0}}(x_{0})\cap \Omega} (z-x_{0} ) dz \bigg\rangle 
\\ & = q|x_{0}-y_{0}|^{-q-1} \bigg\langle   \frac{x_{0}-y_{0}}{|x_{0}-y_{0}|},\beta(x_{0})  \bigg\rangle s_{\eps}(x_{0}) + O(\eps s_{\eps}(x_{0}))
\ge C_{1} s_{\eps}(x_{0}),
\end{split}
\end{align}
and
\begin{align*}
&\kint_{E_{\eps}^{ \beta ,x_{0}}(x_{0}) \cap \Omega} \langle D^{2}\phi (x_{0})(z-x_{0}),z-x_{0} \rangle  dz
\\ & =q|x_{0}-y_{0}|^{-q-4} \times \\ & \quad \bigg\langle (q+2)(x_{0}-y_{0})\otimes
(x_{0}-y_{0})-|x_{0}-y_{0}|^{2}I_{n} : \kint_{E_{\eps}^{ \beta ,x_{0}}(x_{0})\cap \Omega} \hspace{-1em} (z-x_{0})\otimes(z-x_{0} )dz\bigg\rangle
\\ & = q|x_{0}-y_{0}|^{-q-4}  \bigg\langle  (q+2)(x_{0}-y_{0})\otimes
(x_{0}-y_{0})-|x_{0}-y_{0}|^{2}I_{n} :  \frac{1}{n+2}V_{\beta(x_{0})}^{-1}\bigg\rangle  \eps^{2} \\ & \quad + O(\eps s_{\eps}(x_{0}))
\\ & = \frac{q|x_{0}-y_{0}|^{-q-4}}{n+2}\bigg( (q+2)\langle V_{\beta(x_{0})}^{-1}(x_{0}-y_{0}),
x_{0}-y_{0}\rangle-\Tr (V_{\beta(x_{0})}^{-1}) |x_{0}-y_{0}|^{2}\bigg)\eps^{2}\\ & \quad + O(\eps s_{\eps}(x_{0})) .
\end{align*}
Since $V_{\beta(x_{0})}^{-1} =(\tilde{a}_{ij})$ is given by 
\begin{equation}a_{ij} =
\left\{ \begin{array}{llll}
|\beta'(x_{0})|^{2}+1 & \textrm{if $i=j\neq n  $,}\\
1& \textrm{if $ i=j=n$,}\\
\beta_{i}(x_{0}) & \textrm{if $ i\neq n, j=n $ or $i=n, j\neq n $}\\
0 & \textrm{otherwise,}\\
\end{array} \right.
\end{equation}
we can calculate that eigenvalues of $V_{\beta(x_{0})}^{-1} $ are $$1 \ \textrm{(multiplicity $n-2$)}, 
\ \frac{|\beta'(x_{0})|^{2}+2 \pm \sqrt{(|\beta'(x_{0})|^{2}+2 )^{2}-4}}{2}$$ and
$\Tr (V_{\beta(x_{0})}^{-1}) =|\beta'(x_{0})|^{2}+n $.
Then, by using
$$\frac{|\beta'(x_{0})|^{2}+2 - \sqrt{(|\beta'(x_{0})|^{2}+2 )^{2}-4}}{2} \ge
\frac{1}{|\beta'(x_{0})|^{2}+2 } ,$$
we have
\begin{align*}
 &(q+2)\langle V_{\beta(x_{0})}^{-1}(x_{0}-y_{0}),
 x_{0}-y_{0}\rangle-\Tr (V_{\beta(x_{0})}^{-1} |x_{0}-y_{0}|^{2}) \\ &
\ge \bigg(\frac{q+2}{|\beta'(x_{0})|^{2}+2}- (|\beta'(x_{0})|^{2}+n )\bigg)|x_{0}-y_{0}|^{2}.
\end{align*}
Recall that 
$|\beta'(x_{0})|^{2}<\frac{1}{\zeta_{0}^{2}} -1.$
We get $\frac{q+2}{|\beta'(x_{0})|^{2}+2}- (|\beta'(x_{0})|^{2}+n )>0 $
if $$q>\bigg(\frac{1}{\zeta_{0}^{2}}+n-1 \bigg)\bigg(\frac{1}{\zeta_{0}^{2}}+1 \bigg)-2 .$$
Hence, for some $C_{1}>0$, we obtain that
\begin{align}\label{comb3}
&\kint_{E_{\eps}^{ \beta ,x_{0}}(x_{0}) \cap \Omega} \langle D^{2}\phi (x_{0})(z-x_{0}),z-x_{0} \rangle  dz \ge C_{1}\eps^{2}.
\end{align}

Combining \eqref{mrge}, \eqref{comb2} with \eqref{comb3}, we get the desired result.
\end{proof}

The following lemma can be proved by using a similar argument of the proof of Lemma \ref{esttaus}
with the following submartingale
$$ |X_{k}-Z^{\rho}(X_{k})|^{-(\zeta_{0}^{-2}+n-1 )(\zeta_{0}^{-2}+1 )}-C_{0}(k\eps^{2}+S_{k}^{\eps,x_{0}})  .$$
\begin{lemma}\label{esttausob}Let $\overline{r}<r $ with $r$ as above.
Assume that $|x_{0}-Z^{\rho}(x_{0})  |>\rho-h$ for $h \in (0, \frac{\overline{r}}{2}-\eps)$ and we fix the strategy $S_{II}^{\ast}$ to pull toward $Z_{\rho}(X_{k})$ for Player II.
Then for every small $\eps>0$, $x_{0}\in \overline{\Omega}$, we have
\begin{align}
\sup_{S_{I}}\mathbb{E}_{S_{I},S_{II}^{\ast}}^{x_{0}}[\eps^{2}\overline{\tau}^{\eps,\rho,h,x_{0}}+S_{\overline{\tau}^{\eps,\rho,h,x_{0}}}^{\eps , x_{0}}] \le C\overline{r}^{-(\zeta_{0}^{-2}+n-1 )(\zeta_{0}^{-2}+1 )-1}(h+\eps)
\end{align}
for some constant $C$ depending on $n, \alpha, \rho $ and $\zeta_{0}$.
\end{lemma}

Since we already have the interior estimate (Theorem \ref{intlip}) and the boundary estimate (Theorem \ref{bdryregob}), we can discuss the convergence of the value function satisfying \eqref{dppobpgen}, which corresponds to Theorem \ref{convrob} for \eqref{probge}. We omit the detail of the proof, but it can be shown by a similar argument in the proof of Theorem \ref{convrob} with the following observation 
\begin{align}\label{testo1}\begin{split} &  \kint_{E_{\eps}^{ \beta ,x}(x)\cap \Omega}\phi(y)dy  
=    \phi(x) -s_{\eps}(x) \langle D\phi(x), \beta(x) \rangle  + O(\eps s_{\eps}(x)  )\end{split}
\end{align} 
for any $x \in \overline{\Omega}$ and $\phi \in C^{2}(\overline{\Omega})$,
which can be derived from Proposition \ref{besob}.
 
\begin{theorem} \label{convob}
Let $u_{\eps}$ be the function satisfying \eqref{dppobpgen} with $\beta \in C(\Gamma_{\eps})$ and $G \in C^{1}(\Gamma_{\eps})$ for each $\eps >0$.
Assume that $\beta(x)=\mathbf{n}(\pi_{\partial \Omega}x)$ for any $x \in  \Gamma_{\eps}\backslash \Gamma_{\eps/2}$
and $\big|\langle\mathbf{n}(\pi_{\partial \Omega}x), \frac{\beta(x)}{|\beta(x)|}\rangle \big| \ge\zeta_{0}$ in $\Gamma_{\eps}$ for some $\zeta_{0} \in (0,1)$.
Then, there exists a function $u: \overline{\Omega} \to \mathbb{R}^{n}$ and
a subsequence $\{  \eps_{i}\}$ such that $u_{\eps_{i}}$ converges uniformly to $u$ on $\overline{\Omega}$ and
$u$ is a viscosity solution to the problem \eqref{probge}.
\end{theorem}


\end{document}